\documentclass{amsart}
\everymath{\displaystyle}  

\usepackage{amsmath}
\usepackage{amssymb}
\usepackage{amscd}
\usepackage{pb-diagram}
\usepackage[all]{xy}
\usepackage[T1]{fontenc}
\usepackage{hyperref}

\newcounter{rem}
\newenvironment{remark}{\addtocounter{rem}{1}
\begin{enumerate} \item[] {\bf Remark~\therem . }}
{\end{enumerate}\bigskip}

\newtheorem{theorem-introduction}{Theorem}
\newtheorem{corollary-introduction}{Corollary}
\newtheorem{definition}[subsubsection]{Definition}
\newtheorem{lemma}[subsubsection]{Lemma}
\newtheorem{corollary}[subsubsection]{Corollary}
\newtheorem{theorem}[subsubsection]{Theorem}
\newtheorem{proposition}[subsubsection]{Proposition}
\theoremstyle{definition}

\newtheorem{example}[subsubsection]{Example}

\def\Chern{\ensuremath{\mathrm{Ch}}}
\def\codim{\ensuremath{\mathrm codim}}

\def\abupt{\ensuremath{\Rightarrow}}

\def\ker{\ensuremath{\mathrm{Ker}}}
\def\Ker{\ensuremath{\mathrm{Ker}}}

\def\Gr{\mathrm Gr}

\def\tr{{\mathrm tr}}

\def\Tr{{\mathrm Tr}}

\def\FF{{\mathbb  F}}
\def\PP{{\mathbb  P}}
\def\NN{{\mathbb  N}}
\def\ZZ{{\mathbb  Z}}
\def\CC{{\mathbb  C}}
\def\RR{{\mathbb  R}}
\def\QQ{{\mathbb  Q}}

\def\HH{{\mathbb  H}}

\def\VN{ \text{M}}
\def\Vn{ \text{M}}
\def\Dom{{\mathrm{Dom}}}
\def\Gal{{\mathrm{Gal}}}
\def\Im{{\mathrm{Im}}}
\def\Ran{{\mathrm{Ran}}}
\def\Hom{{\mathrm{Hom}}}

\def\p{{\ensuremath{ p_{*(2)} } }}

\def\Id{\ensuremath{{\rm Id}}}

\def\U{\ensuremath{{\mathcal{U}}}}
\def\C{\ensuremath{{\mathcal{C}}}}

\def\UC{{\mathrm UC}}

\def\M{\ensuremath{{\mathcal{M}}}}
\def\O{\ensuremath{{\mathcal{O}}}}

\def\F{\ensuremath{{\mathcal{F}}}}

\def\H{\ensuremath{{\mathcal{H}}}}

\def\B{\ensuremath{{\mathcal{B}}}}

\newcommand{\Adh}[1]{\ensuremath{\overline{#1}}}

\newcommand{\dlog}[2]{\ensuremath{\Omega^{#1}_{X}(\log #2)}}

\def\Ch{\ensuremath{{\rm Ch}}}
\def\Todd{\ensuremath{{\rm Todd}}}

\def\Gl{{\mathrm Gl}}

\def\dbar{\ensuremath{{\overline{\partial}}}}

\def\C{\ensuremath{{\mathcal{C}}}}

\def\Rank{{\rm Rank}}
\def\Tr{{\rm Tr}}

\newcommand{\Trace}[1]{\ensuremath{{\rm \ Tr}#1}}

\def\End{{\rm End}}

\def\Vol{{\rm Vol }}
\def\tbullet{{\mbox{\tiny{$\bullet$}}}}

\def\Higgs{{\mathrm{Higgs}}}

\title[]{Application of Cheeger-Gromov theory to $l^{2}-$Cohomology of harmonic Higgs bundles over covering of finite volume complete manifolds}
\author{Pascal Dingoyan}
\address{Institut Mathématique de Jussieu\\
4, place Jussieu\\
75005 Paris.}
\email{pascal.dingoyan@imj-prg.fr}

\author{Georg Schumacher}
\address{Fachbereich Mathematik und Informatik, Philipps-Universität Marburg, Lahnberge, Hans-Meerwein-Stra{\ss}e, D-35032 Marburg, Germany}
\email{schumac@mathematik.uni-marburg.de}

\begin{document}
\bibliographystyle{plain}
\begin{abstract} We review and apply Cheeger-Gromov theory on $l^{2}-$cohomology of infinite coverings of complete manifold with bounded curvature and finite volume. Applications focus on $l^{2}-$cohomology of (pullback of) harmonic Higgs bundles on coverings of Zariski open sets of Kähler manifolds. The $l^{2}-$Hodge to DeRham spectral sequence of these Higgs bundles is seen to degenerate at $E_{2}$.
\end{abstract}
\maketitle

\section{introduction}
\subsection{}
In a series of articles, J.Cheeger and M.Gromov (\cite{CheGro85-bounds},\cite{CheGro85-characteristic}....) extend Atiyah's theory of $l^{2}-$Betti numbers of coverings of compact manifolds to coverings of complete manifolds of finite volume and bounded curvature. Let $p:(X,g)\to (Y,g_{0})$ with $g=p^{*}(g_{0})$  be a Galois covering of a Riemannian manifold of finite volume $(Y,g_{0})$ such that $(X,g)$ is of bounded curvature and positive injectivity radius. Let $\Gamma$ be the Galois group of $p$. Then
\begin{theorem}[Cheeger-Gromov]\label{cheeger-gromov}
\begin{itemize}
\item[i)] The unbounded operator
$$
d+d^{*}:\oplus_{i} L^{2}(X,g,\Lambda^{i}T^{*}(X))\to \oplus_{i} L^{2}(X,g,\Lambda^{i}T^{*}(X))
$$
is $\Gamma-$Fredholm, and the Murray-Von Neumann $\Gamma-$dimensions of the harmonic spaces
$
\H^{i}_{(2)}(X,g)=\Ker(d+d^{*})_{|L^{2}(X,\Lambda^{i}T^{*}(X))}
$
are finite.
\item[ii)] The $l^{2}-$Euler characteristic $
    \chi_{(2)}(X)=\sum_{i}(-1)^{i}\dim_{\Gamma}\H^{i}_{(2)}(X)
    $
    is expressible in term of the $\hat A-$genus on $Y$ (Gauss-Bonnet integral):
\begin{align*}
\chi_{(2)}(X)=\int_{Y}\hat A(Y,g_{0})\,.
\end{align*}
\item[iii)] The $l^{2}-$Betti numbers $b^{i}_{(2)}(X,g)=\dim_{\Gamma}\H^{i}_{(2)}(X,g)$ are homotopical invariants.
    \end{itemize}
\end{theorem}
We recall that the $\Gamma-$equivariant operator $A: L^{2}(X,E,h)\to L^{2}(X,F,h')$ is $\Gamma-$Fredholm if the Schwartz kernels of $1_{[0,\epsilon]}(AA^{*})$ and $1_{[0,\epsilon]}(A^{*}A)$ are integrable on a fundamental domain of $p$, for some $\epsilon>0$.

\medskip

In this article, generalisations of these statements are discussed for bundles of bounded curvature  and applied to study the $l^{2}-$cohomology groups of holomorphic bundles or harmonic Higgs bundles (cf.\ \ref{definition of harmonic higgs bundle}) over coverings of Zariski open sets of compact Kähler manifolds.
\medskip

Our primary interest is related to property $(i)$. Over a Kähler manifold, the $\Gamma-$Fredholm property of the Dirac operator $\dbar+\dbar^{*}$ made the statement of a Galois $\partial\dbar-$lemma possible:
\begin{theorem}\label{ddbar introduction}
Let $p:(X,\omega)\to (Y,\omega_{0})$, $\omega=p^{*}(\omega_{0})$, be a Galois covering of a complete Kähler manifold $(Y,\omega_{0})$ of finite volume and bounded curvature. Assume that the injectivity radius of $(X,\omega)$ is positive.
%
Then for any square integrable $d-$closed $(p,q)-$form $\alpha$ which is orthogonal to the space of square integrable harmonic forms, there exists an injective element $r$ of the von Neumann algebra of $M(\Gamma)$, and there exists  a square integrable $(p-1,q-1)-$form $\beta$ such that
\begin{align} 
\partial\dbar \beta=r.\alpha\,.
\end{align}
\end{theorem}
On a compact Kähler manifold, the $\partial \dbar-$lemma is fundamental for the  development of Hodge theory (\cite{Voi,DelGriMorSul}), and Mixed Hodge theory of Deligne (\cite{Deldeux,Deltrois}). Applications in the context of $l^{2}-$Hodge theory were given in \cite{Din2013,Din2018}.

Property $(i)$ will be generalised to a pullback by $p:X\to Y$ of any Dirac bundle $(S,D,h)\to Y$ of bounded curvature or to a pullback of a harmonic Higgs bundle $(E,\dbar,\theta,h)\to Y$ with bounded Higgs field. For the latter class, an analogue of the Galois $\partial\dbar-$lemma is stated in section \ref{ddbar lemma for higgs bundles}. As in the compact case (see \cite{Zuc,Sim92}), we have an isomorphism between DeRham and Dolbeault $l^{2}-$cohomology up to a twist by $\U(\Gamma)$, the ring of operators affiliated to $\VN(\Gamma)$  (see \cite{MurVon,Luc} for the definition, and Section~\ref{tensor product with affiliated operators} for a motivation).
In the sequel we will denote a Riemannian bundle and its pullback under $p$  by the same letter. 

\begin{theorem} With the hypotheses of Theorem \ref{ddbar introduction}, let $(E,\dbar,\theta,h)\to Y$ be a harmonic Higgs bundle with bounded Higgs field  and flat connection $\nabla$  (\ref{definition of harmonic higgs bundle}). The Hodge to DeRham spectral sequence
\begin{align}
(E_{1}^{p,q},d_{1})=(H^{p,q}_{\dbar(2)}(X,E),\theta)\otimes_{\VN(\Gamma)} \U(\Gamma)\abupt \, H^{p+q}_{\nabla(2)}(X,E)\otimes_{\VN(\Gamma)}\U(\Gamma)
\end{align}
degenerates at $E_{2}$.
\end{theorem}

If  $Y$ is a compact manifold, the spectral sequence  $$H^{i}(Y,\p\H^{j}(\theta))\abupt \HH^{i+j}(Y,\p(E\otimes \Omega^{.},\theta))$$ (where the functor $\p$ is defined in (\ref{l2 direct image})) proves vanishing and non-vanishing theorems in the range related to the dimension of the homology sheaves $\H^{j}(\theta)$ of $\theta$. 
 One deduces the following theorem, which implies the generic vanishing theorem of Green-Lazarsfeld \cite{GreLaz}:
\begin{theorem}
Let $p:X\to Y$ be a Galois covering of a compact Kähler manifold. Let $E\to Y$ be a flat unitary bundle and let $\theta$ be a holomorphic $1-$form on $Y$ such that $p^{*}(\alpha)$ is exact. Then
$$
H^{i}_{\dbar (2)}(X,E\otimes \Omega^{j})\otimes_{\VN(\Gamma)}\U(\Gamma)=0 \quad\text{for}\quad|i+j-n|>\dim Z(\alpha).
$$
\end{theorem}
The case of a punctured curve implies that a similar statement for coverings of Zariski open subsets of Kähler manifolds is false.

\medskip

We next discuss quickly the point $(ii)$. The heat equation proof  of $(ii)$ needs control on the remainder terms in the asymptotic expansion of the heat kernel $e^{-t\Delta}$ as $t$ goes to zero. Therefore, we add the ad hoc hypothesis that the pullback of the Dirac bundle is of bounded geometry in the sense of Definition~\ref{manifold of bounded geometry} 
and state an index theorem (see Theorem~\ref{index theorem}).

The characteristic integrals are interpreted in the context of a Zariski open set $Y=\bar Y\setminus D$ where $D$ is a normal crossings divisor in a compact Kähler manifold $Y$.
It is well known (\cite{CorGri,KobRyo,TiaYau,Zuc}) that there exists a complete Kähler metric $\omega_{\bar Y,D}$ of finite volume on $Y$, whose covariant derivatives of any order of its Riemannian tensor are bounded. Moreover for   suitable coverings $p:X\to Y$, the pull back metric $\omega=p^{*}(\omega_{\bar Y,D})$ will have a positive injectivity radius. We call such a covering a Poincaré covering.
\begin{theorem}Let $\bar E\to \bar Y$ be a holomorphic vector bundle on the compact Kähler manifold $\bar Y$. Assume $E:=\bar E_{|Y=\bar Y\setminus D}\to Y$ is equipped with a Hermitian metric $h$, which is good in the sense of Mumford (\cite{Mum}, cf.\ Sec.~\ref{Interpretation of the characteristic integrals}). Assume furthermore that the pullback bundle is of bounded geometry.
Let  $p:(X,\omega)\to (Y,\omega_{\bar Y,D})$ be a Poincaré Galois covering. Let $T(\bar Y)(-\log(D))$ be the logarithmic holomorphic tangent bundle with respect to $D$ (which is the dual of the logarithmic holomorphic cotangent bundle). Then
\begin{align*}
\chi_{(2)}(X,E,h,\dbar) &:=\sum_{i}(-1)^{i}\dim_{\Gamma}\H^{(0,i)}_{\dbar(2)}(X,E,h)\\
                                          &=\int_{\bar Y}\Todd[\,T(\bar Y)(-\log(D))\,]\Ch(\bar E)\,.
\end{align*}
On deduces the $l^{2}-$Euler characteristic of $p:(X,\omega)\to Y$ is equal to the logarithmic Euler characteristic of $(\bar Y,D)$:
 \begin{align*}
\chi_{(2)}(X,\CC,d) =\sum_{j}(-1)^{j}\chi_{(2)}(X,\Omega^{j}(X),\dbar)
                                          =\int_{\bar Y} c_{n}(T(\bar{Y})(-\log(D)))=\chi(\bar Y\setminus D)\,.
\end{align*}
\end{theorem}
For a certain restricted class of harmonic bundles, the following fact can be shown:
\begin{theorem} Let $p:X\to Y$ be a Poincaré covering as above. Let $(E,\dbar,\theta,h)\to Y$ be a tame nilpotent harmonic bundle then
\begin{align*}\chi_{(2)}(X,(E,h)\otimes \Omega^{p}_{X},\dbar)=r\int_{\bar Y}\Todd (T\bar Y(-\log(D)))\Chern(\Omega^{p}_{\bar Y}(\log(D)))\,.\end{align*}
\end{theorem}

We refer to e.g. \cite{BalBruCar},\cite{Zuc},\cite{Lot} for comparisons between the characterictic integrals with the index of various Dirac operators on $Y$.
\medskip

The invariance by homotopy of $l^{2}-$Betti numbers ((iii) above) is a deep property in the context of the Cheeger-Gromov theory. It is a result of Dodziuk \cite{Dod} in the compact case. It was generalized to $l^{2}-$torsion by Gromov and Shubin in \cite{GroShu}.
We will take some care to expose its proof in the following case:
\begin{theorem}\label{homotopy invariance introduction}
Let  $p:(X,g)\to (Y,g_{0})$, $g=p^{*}(g_{0})$, be a Galois covering of a Riemannian manifold $(Y,g_{0})$ of finite volume such that $g$ is of bounded curvature of positive injectivity radius. Let $(E,h,\nabla_{h}, \nabla)\to (Y,g_{0})$ be a Riemannian bundle of bounded curvature equipped with a flat connection $\nabla$. Assume that it satisfies property $AC^{0}$ (cf.\ Definition~\ref{property A})\footnote{It is satisfies in the Kähler case and should be true in general.}. Let $\underline{E}=\Ker(\nabla)$ be the locally constant sheaf that it defines. Let $ Y_{0}\subset Y$ be a compact submanifold with smooth boundary, $\dim  Y_{0}=\dim Y$, such that $  Y_{0}\to Y$ is a homotopy equivalence.

Let $K(Y_{0})$ be a simplicial structure on $Y_{0}$, and  let $p^{-1}[K(Y_{0}), \underline{E}]$ be the pullback simplicial complex with coefficients in the pull-back system. Then:
\begin{align}\label{equality introduction}
b^{i}_{(2)}(X,E,\nabla, h)=b_{(2)}^{i}(p^{-1}[K(Y_{0}), \underline{E}])\,.
\end{align}
%
Moreover the $l^{2}-$Betti numbers are homotopy invariant, and are computable in terms of $l^{2}-$sim\-pli\-cial cohomology.
\end{theorem}
Homotopy invariance of the $l^{2}-$Betti numbers implies that the above theorem holds for any compact submanifolds $Y_{0}$.
On one hand, the $l^{2}-$Betti numbers satisfy Poincaré-Hodge duality, as they are metric invariants, on the other hand, they vanish in degree not less than the homotopical dimension of $Y$. Hence:
\begin{corollary} Moreover, assume that $Y$ has the homotopy type of a finite $CW-$complex of dimension $k$, then $b^{i}_{(2)}(X,E,\nabla, h)=b_{(2)}^{i}(p^{-1}[K(Y_{0}), \underline{E}])$ is non-vanishing only in the range  $\dim_{\RR}Y-k\leq i\leq k$.  In particular:

\item[i)] If $2k<\dim Y$, then any Galois covering of positive injectivity radius is $l^{2}-$acyclic.
\item[ii)] Assume $Y$ is a Stein manifold of dimension $n$, and let $(E,\dbar,\theta,h)\to Y$ be a harmonic Higgs bundle, with bounded Higgs field, then the homology of
$$ \ldots\overset{\theta\otimes 1}{\to}H^{p-1,q}_{\dbar(2)}(X,p^{*}(E))\otimes \U(\Gamma) \overset{\theta\otimes 1}{\to}H^{p,q}_{\dbar(2)}(X,p^{*}(E))\otimes \U(\Gamma) \overset{\theta\otimes 1}{\to}\ldots
$$
is vanishing if $p+q\not=n$.
\end{corollary}
Hence duality allows to deduce a vanishing theorem below the homotopical dimension (the local system $\p(\underline{E})$ is defined in (\ref{l2 direct image})):
\begin{corollary}
Let $\bar Y$ be a compact Kähler manifold, $D$ a normal crossings divisor such that $Y=\bar Y\setminus D$ is Stein. Let $p:X\to \bar Y\setminus D$ be a Poincaré covering.
Let $(E,\nabla)\to \bar Y\setminus D$ be a semi-simple unipotent flat bundle. Then $H^{i}(Y,\p(\underline{E}))\otimes_{\VN(\Gamma)} \U(\Gamma)=0$ if $i\not=\dim_{\CC}Y$.
\end{corollary}

\medskip
Concluding remarks:
A large part of this article is expository, the only new point concerns the degeneracy of the Hodge to DeRham spectral sequence for the $l^{2}-$cohomology of Higgs bundles on covering spaces of open manifolds.
The exposition given by Cheeger and Gromov uses numerous results spread in four different papers (\cite{CheGro85-bounds},\cite{CheGro85-characteristic},\cite{CheGro-group}, \cite{CheGro91-chopping}). Some parts of this material were discussed by other authors, in particular by Lück, Lück-Lott, and Schick (\cite{LotLuc,Luc,LucSch,SchDissertation}). However, we tried to organise the material in a single linear short paper with emphasises on the main implications and subtleties.

These points understood, we tried to apply the Cheeger-Gromow theory to harmonic Higgs bundles on covering of Zariski open sets in Kähler manifolds. It has a drawback in this context. Assuming the pullback metric of the covering $p:X\to Y$ is of positive injectivity radius and bounded curvature (i.e.\ $p$ is of Poincaré type) seems to be either a strong constraint or not easy to check. However elementary examples show that the previous results are not valid for arbitrary Galois coverings of Zariski open sets.

In a forthcoming article, the first named author will develop another model which allows the study of any Galois covering $p:X\to Y$ of a Zariski open subset of a compact Kähler manifold.

The first named author addresses his thanks for the highly profitable working days  and  nice hospitality at the Marburg Institute. Many thanks also go to
the team Géometrie of Institut Elie Cartan, to the teams Analyse complexe et Géometrie and Topologie et Géometrie algébrique of Institut mathématiques de Jussieu
for the numerous discussions and explanations on the subject.
The second named author expresses his thanks for the kind hospitality and perfect working conditions at Institut mathématiques de Jussieu.

\section{preliminaries}
\subsection{Manifolds of bounded geometry}\label{manifold of bounded geometry}
\begin{definition}[\cite{Kor,ShuNantes,SchNachrichten}]
\begin{itemize}
\item[1)] A Riemannian manifold $(X,g)$ is of bounded geometry if
\begin{itemize}
\item[a)]The injectivity radius $r_{inj}(X)$ is strictly positive,
\item[b)] all covariant derivative of its Riemannian tensor are bounded, $||\nabla^{k}R||_{\infty}\leq C_{k}$, $k\in \NN$.
This last condition is equivalent to: Let $r\in ]0,r_{inj}(X)[$ and $y: U_{x,r}\to \RR^{n}$ and $y': U_{x',r}\to \RR^{n}$ be two domains of canonical coordinates. Then $y'\circ y^{-1}: y(U_{x,r}\cap U_{x',r})\to \RR^{n}$ satisfies \begin{align*}\forall\alpha\in \NN,\, \exists C_{\alpha,r}\, \text{s.t. } \,\forall x,x'\in X, \,\,|\partial^{\alpha}_{y}(y'\circ y^{-1})|\leq C_{\alpha,r}\,.
\end{align*}
\end{itemize}
\item[2)]  A Riemannian manifold with boundary $(\bar X,g)$ is of bounded geometry if
\begin{itemize}
\item[a)] $(\partial X,g_{|\partial X})$ is of bounded geometry,
\item[b)] there exists $r_{c}>0$ such that $e:\partial X\times [0,r_{c}[\to \bar X\quad\quad (x,t)\mapsto \exp_{x}(t\nu_{x})$ is a bilipschitz diffeomorphism, here $\nu_{x}$ is the unit inward normal vector at $x\in \partial X$,
\item[c)] the injectivity radius $i_{x}>r_{inj}>0$ for all $x\in X\setminus e(\partial X\times [0,r_{c}[)$ and $(b)$ above is true for $x\in X:=\bar X\setminus \partial X$,
\item[d)] the second fundamental form $l$ of $\partial X$ has all its covariant derivatives bounded: $\forall k\in \NN,\exists C'_{k}$ such that $||\nabla_{\partial X}^{k} l||_{\infty}\leq C'_{k}$.
\end{itemize}
\end{itemize}
\end{definition}

Let $r\leq \min{(r_{inj}(X),r_{inj}(\partial X))}$. Let $x\in \partial X$, choose an orthonormal frame in $T_{x}(\partial X)$,  let $\kappa_{x}: B_{\RR^{n-1}}(0,r)\times [0,r[\to X$ be given by $(v,t)\mapsto \exp_{\exp_{x}(v)}(t\nu_{\exp(v)})$. If $x\in X$, choose an orthonormal frame in $T_{x}(X)$, let $\kappa_{x}: B_{\RR^{n}}(0,r)\to X$ be $v\mapsto \exp_{x}(v)$. Let us denote by $U_{x,r}$ the image of $\kappa_{x}$.
Then
\begin{theorem}(loc.cit.)\label{uniform partition of unity}
For any $r>0$, there exists a covering $U(r)=\{U_{x_{i},r},\, x_{i}\in \bar X, i\in \ZZ\}$ of $\bar X$ such that $i\geq 0\Rightarrow x_{i}\in \partial X$, $i<0\Rightarrow x_{i}\in X\setminus e(\partial X\times \frac{2r}{3})$, $U(\frac{r}{2})$ is a covering of $\bar X$, there exists $(\theta_{i})_{i\in \ZZ}$ a subordinate uniform partition of unity for $\bar X$:
\begin{align}
\forall k\in \NN,\exists A_{k}\geq 0 &\text{  s.t.   }\quad (\forall i\in \ZZ)\quad ||\nabla_{g}^{k}\theta_{i}||_{\infty}\leq A_{k}\,,\label{boundness}\\
 \forall s>0,\, \exists M_{s}\in \NN & \text{ s.t. }\forall x\in \bar X, \quad  \sharp \{i\in \ZZ\text{ with }supp(\theta_{i})\cap B(x,s)\not=\emptyset\}\leq M_{s}\,.\label{multiplicity}
\end{align}
\end{theorem}

\begin{definition}(loc.cit.) Let $(\bar X,g)$ be a manifold of bounded geometry.
\item[i)] A bundle $E\to X$ is of bounded geometry if its has trivializations $t_{x_{i}}: E_{|U_{x_{i},r}}\to U_{x_{i},r}\times F$ on a covering $\{U_{x_{i},r}\}_{i\in I}$ by coordinates charts as above such that the transition functions $g_{x_{i},x_{i'}}=t_{x_{i'}}\circ t^{-1}_{x_{i}}$ satisfy
\begin{align*}\forall\alpha\in \NN,   \,   \exists C_{\alpha,r}    \,    \text{s.t. } \forall i,i'\in I,   \,\,   |\partial^{\alpha}_{y}g_{x_{i},_{i'}}|\leq C_{\alpha,r}\,.\end{align*}
\item[ii)] A hermitian bundle $(E,h)\to (\bar X,g)$ is of bounded geometry, if $E\to X$ is of bounded geometry, and for any $k\in \NN$, the component of the derivatives, up to order $k$ of the matrices of $h$, in the above trivializations, are uniformly bounded with respect to $i\in I$.
\end{definition}
Hence any tensor bundle on a manifold of bounded geometry is of bounded geometry.
\begin{proposition}[{see \cite[Prop.\ 2.5]{Roe}, \cite[Appendix A]{AtiBotPat}}]
Let $(E,h)\to (\bar X,g)$ be a hermitian vector bundle on the manifold of bounded geometry $(X,g)$. Then $(E,h)$ is of bounded geometry, iff all the covariant derivatives of the curvature tensor of $h$ are bounded.
\end{proposition}

The existence of  a uniform partition of unity implies a Sobolev embedding property in terms of a metric connection $\nabla_{h}$ of a hermitian bundle $(E,h)$ of bounded geometry: If $s\geq\frac{n}{2}$, then
\begin{align}\label{sobolev embedding}
m_{s,k}: H^{s+k}(\bar X,E)\to \UC^{k}(\bar X,E)=\{a\in C^{k}(\bar X,E): \forall 0\leq i\leq k, \,||\nabla^{i}_{h}\alpha||_{\infty}<+\infty\}
\end{align}
is defined and bounded (see \cite{SchNachrichten}, \cite{Roe}, \cite{ShuNantes} or \cite[Chap.\ 4 Cor.\ 1.4 or Prop.\ 4.3]{Tay}).

\subsection{Approximation by metrics of bounded geometry}\label{approximation by bounded geometry}
Concerning the existence of manifolds of bounded geometry, we quote the following important approximation theorem. (see \cite[Thm.\ 2.5]{CheGro85-characteristic}). We use the formulation given by Shi \cite{Shi} for the uniqueness of the evolution equations (proved in \cite{CheZhu}) implies the preservation of the symmetries of the initial metric:
\begin{theorem}\label{smoothing of metric}
Let $(M,g)$ be a complete non compact Riemannian manifold with bounded Riemannian curvature tensor $|| R_{ijkl}||_{\infty,M}\leq k_{0}$. Then there exists $T>0$ such that the evolution equation $\partial_{t} g_{ij}=-2R_{ij}$ with initial data $g$, admits a unique solution $g_{t}$ on $M\times [0,T]$ such that $||\nabla_{g(t)}^{m} R_{ijkl}(t)||_{\infty,(M, g(t))}\leq c(m)t^{-m}$, $m\in\NN$, and $  e^{-2n\sqrt c_{0}t} g\leq  g(t)\leq e^{+2n\sqrt c_{0}t} g$, where  $c_{0}=||R(t)||_{\infty, M\times [0,T]}$ is a finite constant.
\end{theorem}
We refer to formula (9) in \cite{Shi} for the latter bounds. Uniqueness implies that if $g$ is invariant by a group of isometries, then $g(t)$ is also invariant. Hence, let  $p:(X,g)\to (Y,g_{0}) $ be a Riemannian covering map with $g$ of bounded curvature and positive injectivity radius, then $g(t)$ is of bounded geometry (because strict positivity of the injectivity radius is preserved by the flow) and descends to a metric $g_{0}(t)$. Moreover it is known that the Ricci flow preserves the Kähler condition.

We need an analogous theorem for vector bundles. However, a general reference seems not to be known, so we make the following definition.

\begin{definition} \label{property A} A Riemannian bundle $(E,h,\nabla_{h})\to (X,g)$ on a manifold of bounded geometry has the property $AC^{k}$ for $1\leq k\in \NN$, if there exists a metric $h'$ with metric connection $\nabla_{h'}$ on $E$, strictly positive constants $c_{1},c_{2}$ such that $(E,h')\to (X,g)$ has bounded geometry, and $c_{1}h\leq h'\leq c_{2}h$, and, if $1\leq k$,  the connection form  $\nabla_{h}-\nabla_{h'}=A_{h,h'}\in T^{*}(X)\otimes \End(E)$ has bounded covariant derivatives up to order $k-1$. Moreover, it is required that the symmetries of $(E,h,\nabla_{h})$ are symmetries of $(E,h',\nabla_{h'})$.
\end{definition}

\begin{example}
\begin{itemize}
\item[i)]
Assume that the holomorphic Hermitian vector bundle $(E,h)\to (X,\omega)$,  on a complete Kähler manifold, has bounded Chern curvature. Then the Hermite-Einstein flow exists and is unique (see e.g.\ \cite{Zha}). Moreover, assuming $\omega$ has bounded geometry, the standard estimates imply that $(E,h(t))\to (X,\omega)$ has bounded geometry if $t>0$.
\item[ii)]
Assume that a further flat connection $\nabla$ is given such that $\nabla-\nabla_{h}=A_{h}$ is bounded. Then $A_{h'}=\nabla-\nabla_{h'}$ is bounded, and from the equation $0=\nabla^{2}_{h'}+[\nabla_{h'},A_{h'}]+A_{h'}^{2}$, one infers that all covariant derivatives of $A_{h'}$ are bounded.
\end{itemize}
\end{example}
\subsubsection{Invariance by change of metric}\label{invariance by change of metric}
 The reduced cohomology groups (cf.\ \ref{notations mayer vietoris}) associated to metrics $(h,g)$ and $(h',g(t))$ with the same flat connection $\nabla$ are isomorphic $\Gamma-$Hilbert spaces.  Hence most statements that hold for these cohomology groups for bundles of bounded geometry will hold under the sole hypothesis of bounded curvature on a manifold of bounded curvature and positive injectivity radius. This is the case for property $(i)$ and $(iii)$ of Theorem~\ref{cheeger-gromov}. However, property $(ii)$ a priori requires the stronger assumption $AC^{1}$: in the Hermitian case, when passing from the metric $h$ to $h'$, the transgression formula $\Chern(E,h)-\Chern(E,h')=dT$ that relates the Chern characters should be established with $T$ bounded.

 We notice that the Bergmann kernels (Schwartz kernel of the orthogonal projection from $L^{2}$ onto the harmonic spaces) associated to different metrics will not be comparable in general, for they depend on higher order jets of the metric. Nevertheless, the integrals of theses kernels on a fundamental domain for the $\Gamma-$action will be equal, because the $\Gamma-$dimension of a $\Gamma-$Hilbert module  (cf.\ Def.~\ref{neumann dimension}) does not depend on its embedding.

\subsection{Example of a metric of bounded geometry and finite volume, Poincaré coverings}
The Poincaré metric $\omega_{P}=i\frac{dz\wedge d\bar z}{|z|^{2}(\log|z|^{2})^{2}}=\frac{i}{2}\partial\dbar \log(\log|z|^{2})^{2}$ on the punctured disc $D(0,1)\setminus\{0\}$ has Gaussian curvature -$1$ and finite volume near the puncture. The pullback of the Poincaré metric $\omega_{P}$ by the universal covering map $p:D(0,1)\to D(0,1)\setminus\{0\}$ is of bounded geometry. This fact motivates the following definition:
\begin{definition}\label{poincare covering} Let $Y_{1}$ be a complex manifold and $D$ a normal crossings divisor. A Poincaré covering for the pair $(Y_{1},D)$ is a covering $p:X\to Y_{1}\setminus D$ such that for any chart $D(0,1)^{n}\to U\subset Y_{1}$  with $D\cap U=\{z_{1}.\ldots z_{k}=0\}$, the connected components of $p^{-1}(U\setminus D)$ are simply connected.
\end{definition}
 In general, we will not specify a pair $(Y_{1}, D)$ more precisely, if the context is clear. A Poincaré covering for $(Y_{1}, D)$ restricts to a covering of bounded geometry over neighborhood of the boundary divisor $D$.

\begin{proposition}\label{Poincare metric}
Let $D=D_{1}+\ldots+D_{p}$ be a normal crossings divisor in a compact Kähler manifold $(\bar Y,\omega)$. Let  $s_{i}$ be a canonical section for $\O([D_{i}])$ (i.e.\ $\{s_{i}=0\}=D_{i}$) and $|.|$ be a smooth metric for $\O([D])$. Assume $|s_{i}|^{2}\leq e^{-\alpha}$ for some $\alpha>0$. Then for $\epsilon>0$ small enough, $$\omega_{\bar Y,D}:=\omega+\epsilon \, dd^{c}\sum_{i=1}^{p}-\log(|s_{i}|^{2}_{i}(-\alpha\log|s_{i}|^{2}_{i}))$$ is a complete Kähler metric of finite volume on $Y=\bar Y\setminus |D|$ such that its pullback to a Poincaré covering $X\to \bar Y\setminus |D|$ is of $C^{\infty}-$bounded geometry.
\end{proposition}
\begin{proof}   We have
\begin{align}dd^{c}(-\log|s_{i}|^{2}_{i})-\alpha\log(-\log|s_{i}|^{2}_{i})= (1-\frac{\alpha}{-\log|s_{i}|^{2}})dd^{c}(-\log|s_{i}|^{2})+\frac{d|s_{i}|^{2}\wedge d^{c}|s_{i}|^{2}}{|s_{i}|^{4}(\log|s_{i}|^{2})^{2}}\,.
\end{align}

It is sufficient to prove that on the Poincaré covering it is of bounded geometry. Following notations of Def.~\ref{poincare covering}, let $c:{D}^{k}\times D^{n-k}\to U\setminus D$\,;  $z\mapsto (\ldots, \exp({\frac{z_{i}+1}{z_{i}-1}}),\ldots,z_{k+1},\ldots, z_{n})$ be the universal covering map.
 By definition any connected component $B_{0}$ of $p^{-1}({D^{*}}^{k}\times D^{n-k})$ is simply connected hence a biholomorphic map $l:D^{n-k}\times D^{k}\to B^{0}$ exists such that $p\circ l=c$. Then one uses the estimates given in \cite[Lemma2, p.405]{KobRyo} or \cite[Lemma 2.1, or p.603--605]{TiaYau}, where the local computations near the divisor do not use the fact that $\omega$ is a Ricci-form.
\end{proof}

\begin{example}
\begin{itemize}
\item[i)]  Let $\rho :\pi_{1}(Y)\to \Gamma$ be a representation such that $\pi_{1}(U\setminus D)\to \pi_{1}(Y)\overset{\rho}{\to} \Gamma$ is injective for any chart $U$ as in definition \ref{poincare covering}. Then the covering associated to $\ker \rho$ is a Poincaré covering. In particular, let $\gamma_{i}$ be a meridian around an irreducible component $D_{i}$ of $D$ and assume $D_{i}$ smooth. Assume that the following implication holds: $D_{i_{1}}\cap \ldots \cap D_{i_{k}}\not=\emptyset\Rightarrow \gamma_{i_{1}},\ldots,\gamma_{i_{k}}$ are rationally independent in $H_{1}(Y_{1},\QQ)$. Then any covering, which dominates the abelian covering is of Poincaré type. As example is the complement of two points in $\PP^{1}$, or of  three lines in general position in $\PP^{2}$.
\item[ii)] The following lemma proves that Poincaré coverings are final for quasi-projective manifolds.
\end{itemize}
\begin{lemma} 
\begin{itemize}
\item[a)] Let $D$ be a normal crossings divisor in a compact Kähler manifold $Y_{1}$. Assume that for any $x\in D$ and any irreducible component $D_{i}$ which contains $x$, there exist effective divisors $L_{i}$, $K_{i}$ such that $D_{i}\leq L_{i}\leq D$, $K_{i}\leq D$, $x\not \in K_{i}$, and $[L_{i}], [K_{i}]$ are proportional classes in $H^{2}(Y,\QQ)$.
Then there exists an Abelian Poincaré covering for the pair $(Y=Y_{1}\setminus D,D)$.
\item[b)]Let $D$ be a strictly normal crossings divisor in the projective manifold $Y_{1}$. Then there exists a finite number of smooth irreducible very ample divisors $H_{1},\ldots, H_{p}$ such that the pair $D'=D\cup_{i}H_{i}$ and  $Y=Y_{1}\setminus D'$ admits an Abelian Poincaré covering. Moreover, the kernel of $\pi_{1}(Y_{1}\setminus D')\to \pi_{1}(Y_{1}\setminus D)$ is central.
\end{itemize}
\end{lemma}
\begin{proof}
{\it a)} Let $D=\cup_{1\leq i\leq r}D_{i}$ be the irreducible decomposition of $D$, let $I\subset \{1,\ldots,r\}$ be chosen such that $D_{I}:=\cap _{i\in I}D_{i}$ is non empty. The divisors $L_{i}$ and $K_{i}$ associated to one point in $ D_{I}\setminus \cup_{j\not\in I}D_{j}$ satisfy the same condition for all other points in this set. Our hypothesis implies that there exists $a,b\in \NN^{*}$ such that $a L_{i}$ is homologous to $b K_{i}$. By classical Hodge theory there exists a logarithmic one form $\alpha_{i,I}$ with residue $a.L_{i}-b. K_{i}$. Since  $K_{i}$ does not intersect $ D_{I}\setminus \cup_{j\not\in I}D_{j}$, the residue of $\alpha_{i,I}$ on $D_{i}, \,i\in I$, is non vanishing.  Let $(z:U\to D(0,1)^{n})$ be a coordinate neighborhood centered at $x\in  D_{I}\setminus \cup_{j\not\in I}D_{j}$ such that $D\cap U=\{z_{1}\cdot\ldots\cdot z_{k}=0\}$. The restriction of the forms $\alpha_{i,I}, \, i\in I$ to $U\setminus D$ generates its cohomology.  The cover of $Y=Y_{1}\setminus D$ associated to the forms $\alpha_{i,I}, i\in I,\, I\subset \{1,\ldots,r\} $, which are closed holomorphic forms on $Y_{1}\setminus D$, defines an Abelian Poincaré covering of $Y\setminus D$ (set $\alpha_{i,I}=0$ if $D_{I}$ is empty).

{\it b)} Let $H$ be a very ample divisor such that $\forall i\in \{1,\ldots,r\},\, H+D_{i}$ is very ample. Let $H_{0},\ldots, H_{n}$ be smooth elements of the linear system of $H$ such that $\cup_{0\leq i\leq n}H_{i}\cup D$ is a normal crossings divisor. Define inductively divisors $L^{i}_{j}, \, 1\leq i\leq r,\, 0\leq j\leq n$ as follows:

The divisors $L^{1}_{0},\ldots, L^{1}_{n}$ are smooth members of the linear system of $H+D_{1}$ such that $\cup_{j} H_{j}\cup D\cup_{j} L^{1}_{j}$ is a normal crossings divisor.

Assume that the divisors $L^{i}_{j}$, $0\leq j\leq n$,  are defined, if $i\leq k$. Then $L^{k+1}_{0},\ldots, L^{k+1}_{n}$ are smooth members of the linear system of $H+D_{k+1}$ such that $\cup_{j} H_{j}\cup D\cup_{1\leq i\leq k+1, 0\leq j\leq n} L^{i}_{j}$ is a normal crossings divisor.


Then $D'=\cup_{0\leq j\leq n} H_{j}\cup D\cup_{1\leq i\leq r, 0\leq j\leq n} L^{i}_{j}$ satisfies $(a)$ above:  Assume $x\in D'$. After renumbering one may assume that $x$ does not belong to $H_{0}\cup_{1\leq i\leq r}L^{i}_{0}$ for any sets of $n+1$ irreducible component of $D'$ have empty intersection. Let $\sim$ denote linear equivalence. If $x\in H_{i}, i>0$, then $H_{i}-H_{0}\sim 0$, if $x\in D_{i}$ then $H_{0}+D_{i}-L^{i}_{0}\sim 0$, if $x\in L^{i}_{k}, k>0$ then $L^{i}_{k}-L^{i}_{0}\sim 0$ . Note moreover that logarithmic derivatives of the rational functions which give the above linear equivalence may be taken as the logarithmic forms( $\alpha_{i,I}$) constructed in the first point.
%

The fact that $\pi_{1}(Y_{1}\setminus D')\to \pi_{1}(Y_{1}\setminus D)$ is central is Nori's theorem (\cite[Cor.\ 2.5 or 2.10]{Nor}).
\end{proof}
\end{example}

\subsection{Sobolev spaces and functional calculus}
\subsubsection{Dirac operator}
%
%
As a general reference on Laplace operator on vector bundles and twisted Dirac operators  and their functional calculi, we will use \cite{BerGetVer},\cite{Roe-elliptic},\cite[Sec.\ 1]{Roe}, (see also\cite{MaMar}).
%

Let $(V, q)$ be an Euclidian space. The Clifford algebra $C(V,q)$ is the algebra generated by $V$ and relations $vw+wv=-2q(v,w)$, hence $v^{2}=-q(v)$. Let $V \subset C(V,q)$ act on $\Lambda V$ by  $c(v).\alpha=v\wedge\alpha-v\lrcorner\alpha$ where $v\in V$, $\alpha\in \Lambda V$, and $\lrcorner$ is the interior product associated to the Euclidean structure. Since $((v\wedge.) (v\lrcorner.)+(v\lrcorner.)(v\wedge.))\alpha=q(v)\alpha$, this action extends to $C(V,q)$ and defines a Clifford module structure on $\Lambda V$.

A Clifford module $E$ of $C(V,q)$ equipped with a metric is called self-adjoint if the operators $c(v)$ with $v\in V$ are skew-adjoint. Therefore, any unit vector operates in a unitary way, because $c(v)c(v)^{*}=q(v)$.

Let $(X,g)$ be a Riemannian manifold, and let $C(X,g)^{\CC}$ be the complexified Clifford algebra bundle of $X$ whose fiber at $x\in X$ is $C(T^{*}_{x}X,g_{x})\otimes \CC$ (here we use the notation of \cite{BerGetVer} and \cite{Tay}, which is different
from \cite{Roe}).  Since $\nabla g=0$,  the Levi-Cevita connection naturally extends to $C(X,g)^{\CC}$.
\begin{definition}
A Clifford bundle $(E,h,\nabla)\to X$ over $X$ is a hermitian vector bundle with a compatible connection and a structure of a left $C(X,g)^{\CC}-$module such that
\begin{itemize}
\item[i)] $\forall x\in X,\forall v\in T_{x}X$, $||v||=1$, $\quad c(v):E_{x}\to E_{x}$ is an isometry,
\item[ii)] $\quad\nabla c(v).s=c(\nabla v).s+c(v).\nabla s$.

 \item[iii)] The Dirac operator associated to the Clifford module is $D:=c\circ \nabla$, where $c:T^{*}X\otimes E\to E$, is given by the Clifford multiplication.

 \item[iv)] We say that  $E$ is $\ZZ_{2}-$graded, if it decomposes as $E=E^{+}\oplus E^{-}$, the decomposition is orthogonal (hence preserved by the metric connection) and the Clifford action is odd: if $\eta$ is the parity operator $\eta_{|E^{\pm}}=\pm Id_{E^{\pm}}$, then \begin{align}\eta\circ c(v)+c(v)\circ \eta=0\iff c(v):E^{\pm}\to E^{\mp}\,.
 \end{align}
\end{itemize}
 \end{definition}
 In a orthonormal basis $(e_{1},\ldots,e_{n})$ for $T_{x}X$, one has $Ds=\sum_{i}  c(e_{i})\nabla s$. If moreover $E$ is $\ZZ_{2}-$graded, then $D$ is odd, it exchanges sections of the positive and negative eigenbundles of $\eta$. In this case $D: C^{\infty}_{0}(X,E^{\pm})\to C^{\infty}_{0}(X,E^{\mp})$ is a generalized Dirac operator: that is the symbol of $D^{2}$ is $\sigma_{2}(D^{2})(\zeta)=|\zeta|^{2}\iff$ in local coordinates $D^{2}=\sum_{i,j}g^{i,j}\partial_{i}\partial_{j}+\text{first order terms}$ (\cite[p.116]{BerGetVer}, \cite[Chap.1]{Tay}).
 %

%
We give the three examples of Dirac operators that will be used in the sequel (\cite[Sec.\ 3.6]{BerGetVer}, \cite[p.49]{Roe-elliptic})
 \begin{trivlist}
 \item[1)] The DeRham operator: $(\Lambda T^{*}X,\nabla_{g}, g)\to X$ is a Clifford modules bundle. It is graded according to the parity of the degree. The associated Dirac operator is \begin{align}D=d+d^{*}:\C^{\infty}_{0}(X,\,\Lambda T^{*}X)\to \C^{\infty}_{0}(X,\,\Lambda T^{*}X)\,.\end{align}
Indeed let $\epsilon:T^{*}X\otimes \Lambda^{.} T^{*}X\to \Lambda^{.+1} T^{*}X$ be the exterior multiplication and $i:T^{*}X\otimes \Lambda^{.} T^{*}X\to \Lambda^{.-1} T^{*}X$ be the interior product morphism. Then $d=\epsilon\circ \nabla$ and $d^{*}=-i\circ \nabla$ for the Levi-Cevita connection $\nabla$ is torsion free.

\item[2)] Let $(E,\nabla,h)\to X$ be a hermitian vector bundle with a hermitian connection, and let \break $(S,\nabla',h')\to X$ be a Clifford bundle. Then $(E\otimes S,h\otimes h',\nabla\otimes 1+1\otimes \nabla')\to X$ is a Clifford bundle. This construction will be extended to hermitian bundles with a flat connection.
 \item[3)] The Dolbeault operator on Kähler manifolds. Let $(X,g)$ be a hermitian manifold. Then $\Lambda T^{*(0,1)}X$ is a Clifford submodule of $\Lambda TX\otimes \CC$.
 Let $(E,h,\nabla)\to (X,g)$ be a hermitian holomorphic vector bundle with Chern connection $\nabla$ over the Kähler manifold $(X,g)$. Then $\Lambda T^{*(0,1)}(X)\otimes E\to X$ with its product structure is a Clifford bundle. Moreover assume the metric $g$ is Kähler, then the associated Dirac operator is $\sqrt2(\dbar+\dbar^{*})$.
 \end{trivlist}

\begin{theorem}\label{functional calculus} Let $(S,h,\nabla)\to (X,g)$ be a Clifford bundle on the complete Riemannian manifold. Then
\begin{itemize}
\item[1)] The associated Dirac operator $D$ is essentially self-adjoint.
\item[2)] Let $\lambda\mapsto E(\lambda)$ be the spectral resolution of $D$. Then $D$ generates the one parameter group of unitary operators $e^{itD}=\int e^{it\lambda}dE(\lambda)$ (Stone's theorem \cite[p.345]{Yos}).
\item[3)] (\cite{Che} Unit propagation speed) Let $s$ be a compactly supported smooth section. The wave equation \begin{align}\frac{\partial}{\partial t}s_{t}=iDs_{t}
\end{align}
with initial data $s_{0}=s$ has a unique solution. Then $s_{t}=e^{itD}s$, $e^{itD}s$ is smooth and compactly supported, and $e^{itD}s(x)=0$ for $|t|< \text{ distance} (x, \text{support}( s))$.
\item[4)]  The mapping \begin{align}
f\mapsto f(D):=\frac{1}{2\pi}\int \hat f(t)e^{itD}dt\,.
\end{align} is a ring homomorphism from $S(\RR)$ to $B(L^{2}(X,S))$ such that $||f(D)||\leq ||f||_{\infty}$. If $f(x)=xg(x)$, then $f(D)=Dg(D)$ (\cite[Chap.\ 9]{Roe}, \cite{CheGroTay}).
\end{itemize}
\end{theorem}

The Sobolev space $W^{k}(S)$ is defined for $k\in \NN$ as the completion of the vector space of smooth sections with compact support, equipped with the norm $||s||_{k}=(\sum_{i=0}^{k}||\nabla^{i}s||^{2})^{\frac{1}{2}}$.  Its dual space $W^{-k}(S):=(W^{k})^{'}$ can be identified with a set of distributional sections.

Let $W^{\infty}=\cap_{k\in \ZZ} W^{k}$ and $W^{-\infty}=\cup_{k\in \ZZ} W^{k}$. The space $W^{-\infty}$ is equipped with its weak topology of the dual space of $W^{\infty}$.
The Schwartz kernel theorem implies that a continuous linear operator from $W^{-\infty}(S)$ to $W^{\infty}(S)$ is represented by a smoothing kernel and the induced morphism $L(W^{-\infty},W^{\infty})\to C^{\infty}(X\times X,S\boxtimes S^{*})$ is continuous (see e.g.\ \cite{Tay}).

\begin{proposition}(see   \cite{ShuNantes}, \cite{Roe})\label{smooth kernel}
 If $(S,h,\nabla)\to (X,g)$ is of bounded geometry, then
\begin{itemize}
\item[0)]  $W^{k}(S)\simeq \Dom(1+\Delta)^{\frac{k}{2}}$ and $s\mapsto (1+\Delta)^{\frac{k}{2}}s$ is an isomorphism from $W^{k}(S)$ to $L^{2}(S)$.
\item[i)] $W^{\infty}(S)$ is continuously embedded in $UC^{\infty}(S)$ (see def. in (\ref{sobolev embedding})).
\item[ii)] A continuous linear operator from $W^{-\infty}(S)$ to $W^{\infty}(S)$ is represented by a smoothing kernel which is uniformly bounded as are all its covariant derivatives.
\end{itemize}
The map $L(W^{-\infty},W^{\infty})\to UC^{\infty}(X\times X,S\boxtimes S^{*})$ is continuous for the topology of bounded convergence.

\end{proposition}

\subsection{\texorpdfstring{$\Gamma-$}{G-}dimension}\label{Gamma-dimension}
\begin{definition}
Let $\VN$ be a Von Neumann algebra on a separable Hilbert space. Let $\VN_{+}$ be the cone of positive operators in $\VN$.
\begin{itemize}
\item[1)]A  trace is a function $t: \VN_{+}\to [0,+\infty]$  such that if $\lambda>0$ and $x,y\in \VN_{+}$, then $t(\lambda x)=\lambda t(x)$, $t(x+y)=t(x)+t(y)$, and for all unitary $u\in A$, $t(u^*xu)=t(x)$.
\item[2)] A trace is normal, if it is continuous on limit of increasing nets.
\item[3)]  A trace is faithful, if  $x\in\VN:\,\varphi(x^*x)=0\Rightarrow x=0$.
\item[4)] A faithful trace is called finite, if $t(1)<+\infty$ (and then $\VN$ is called a finite von Neumann algebra). It is called semi-finite, if $\VN_{+}^{t}=\{ y\in \VN_{+}:\, t(y)<+\infty\}$ is weakly dense in $\VN_{+}$. Then for all $x\in\VN_{+}$, $t(x)=\sup_{y\leq x,\ y\in \VN_{+}^{t}}t(y)$.
\end{itemize}
\end{definition}
\begin{example}
\begin{itemize}
\item[1)] Let $\Gamma$ be a discrete group.
Let $\delta_{e}\in l^2(\Gamma)$ be the Dirac function at the unit element $e$ of $\Gamma$. The trace of $n\in \VN(\Gamma)$  is $\tr_{\VN(\Gamma)} n:=<n(\delta_{e}),\delta_{e}>$.
\item[2)] Let $H$ be an infinite dimensional Hilbert space. Then there exists a unique normal semi-finite trace $Tr_{H}$ defined on $\B(H)$ that takes value $1$ at each one dimensional projection. Let $(e_{n})_{n\in \NN}$ be an orthonormal basis of $H$ then (see \cite[Remark 8.5.6]{KadRin})
 \begin{align}
 \forall A\in \B(H)^{+},\quad Tr_{H}A=\sum_{n\in \NN}(Ae_{n},e_{n})\,.
\end{align}

\item[3)] The space of $\Gamma-$invariant bounded operators on the Hilbert space $H\hat\otimes l^{2}(\Gamma)$ is isomorphic to the tensor product $\B(H)\bar\otimes\VN(\Gamma)$. Then $Tr_{H}\otimes \tr_{\VN(\Gamma)}$ defines a semi-finite trace $\Tr_{\Gamma}$ on $\B(H)\bar\otimes\VN(\Gamma)$ (see \cite[Chap.\ IV]{Takesaki}): If $t\in \B(H)\bar\otimes \VN(\Gamma)$ is a positive element represented (in an orthonormal basis) by an infinite matrix $(n_{ij})_{i,j}$ of elements in $\VN(\Gamma)$, then $\Tr_{\Gamma} (T)=\sum_{i}\tr_{\VN(\Gamma)} n_{ii}$.
\end{itemize}
\begin{definition}\label{neumann dimension}
\begin{itemize}
\item[i)](\cite{Shu},\cite[p.318]{Takesaki}) The ideal of $\Gamma-$trace operators on $H\hat\otimes l^{2}(\Gamma)$ is the set of all finite linear combinations of positive $\Gamma-$invariant operators $A$ such that $\Tr_{\Gamma}(A)<+\infty$.
\item[ii)] Let $L$ be a Hilbert space with a unitary representation of $\Gamma$. Let $i:L\to H\hat\otimes l^{2}(\Gamma)$ be a $\Gamma-$invariant embedding and let $A_{L}$ be the orthogonal projection onto $i(L)$. Then $Tr_{\Gamma}(A_{L})$ is independent of the embedding $i$. It is by definition equal to $\dim_{\Gamma}L$.
\end{itemize}
\end{definition}
We resume the above example.
\begin{itemize}
\item[4)] Let $p: (X,g)\to (Y,g)$ be a Galois covering with Galois group $\Gamma$. Let $(S_{Y},h_{Y})\to Y$ be a Riemannian bundle and let $(S,h)\to X$ be the pull back bundle. Let $F$ be a fundamental domain for the action. Let $(e_{n})_{n\in \NN}$ be an orthonormal basis of $l^{2}(F,S_{|F})$. Then $(e_{n}\otimes \delta_{g})_{n\in \NN,g\in \Gamma}$ is an orthonormal basis of $l^{2}(X,S)$ and induces an isomorphism
\begin{align}
l^{2}(X,S)=l^{2}(F,S_{|F})\hat\otimes l^{2}(\Gamma)\,.
\end{align}
Let $A\in \B(l^{2}(X,S))$ be an operator on $l^{2}(X,S)$ equivariant for the $\Gamma-$action. The above isomorphism decomposes $A$ as $(A_{i,j})_{i,j\in \NN}$ with $A_{i,j}$ a bounded $\Gamma-$equivariant operator on $l^{2}(\Gamma)$: The bilinear form $A_{i,j}: l^{2}(\Gamma)\times l^{2}(\Gamma)\to \CC$ defined by $A_{i,j}(f,g)=(A(e_{i}\otimes f),e_{j}\otimes g)$ is continuous, hence given by $A_{i,j}\in \B(l^{2}(\Gamma))$. Assume moreover that $A$ is $\Gamma-$equivariant, then $A_{i,j}$ belongs to $\VN(\Gamma)$.

If $A$ is positive then
\begin{align}
\tr A_{ii}&=(A e_{i}\otimes \delta_{e}, e_{i}\otimes \delta_{e})=\int_{F}(Ae_{i}(x),e_{i}(x))dV_{g}(x)\,,\\
\Tr_{\Gamma}(A)&:= \left( \Tr_{ L^{2}(F,S_{|F}) }\otimes \tr \right)(A)=\int_{F}\sum_{i\in \NN}(Ae_{i}(x),e_{i}(x))dV_{g}(x)\,.
\end{align}
is equal to the integral of the pointwise trace.
\item[5)]
Let $A: L^{2}(X,S)\to L^{2}(X,S)\cap L^{\infty}(X,S)\cap C^{\infty}(X,S)$  be a smoothing operator
with kernel $k\in C^{\infty}(X\times X,S\boxtimes S^{*})$ (see e.g.\ estimates  (\ref{heat kernel domination}),(\ref{LiYau estimate}) below). Then
\begin{align}
\Tr_{\Gamma}(A)=\int_{F}\Tr_{End(S)}k(x,x)dV_{g}(x)\,.
\end{align}
\end{itemize}

We refer to  \cite[Sec.\ 2]{Ati} or \cite[Chap.\ 15]{Roe} for further details.
\end{example}

\subsection{\texorpdfstring{$\Gamma-$}{G-}Fredholm operators}
\begin{definition}
\begin{itemize}
\item[1)](\cite[Chap.\ XVI]{MurVon})
Let $\VN$ be a finite Von Neumann algebra on $H$.
 A closed densely defined operator $h: \Dom(h)\to H$ is said to be affiliated to $\VN$, if it commutes with $\VN'$: for all unitary $u\in \VN'$, $u\Dom(h)=\Dom(h)$ and $uh=hu$.
 \item[2)] (\cite{CheGro85-bounds,Shu,Luc})
 \begin{itemize}
 \item[i)] A self-adjoint positive affiliated operator $h=\int \lambda dE(\lambda)$ on $H$ is said to be \hfill\break $\VN-$Fredholm, if for some $\lambda_{0}>0$, the spectral projection $E(\lambda_{0})$ satisfies $\Tr_{\Gamma}E(\lambda_{0})<+\infty$.
 \item[ii)] Let $f:\Dom(f)\subset U\to V$ be a closed densely defined unbounded operator, which is $\Gamma-$invariant. Then $f$ is called $\Gamma-$Fredholm, if the self-adjoint operators $ff^{*}$ and $f^{*}f$ are $\Gamma-$Fredholm.
\end{itemize}
\end{itemize}

 \end{definition}
 Note that $h$ is $\Gamma-$Fredholm, iff the bounded operator $h(1+h)^{-1}$ is $\Gamma-$Fredholm. The hypothesis in $ii)$ implies that, for $\epsilon>0$ small enough, the bounded operator $g:=f1_{[\epsilon,\epsilon^{-1}]}(|f|)$, which defines an isomorphism from $\Ker(g)^{\perp}$ to $\Ker(g^{*})^{\perp}$, satisfies  that $\Ker(g)$ and $\Ker(g^{*})$ have finite $\Gamma-$dimension. Hence $f$ is $\Gamma-$Fredholm, iff it is boundedly invertible up to $\Gamma-$trace class operators.

This remark implies the following lemma  which takes care of changes of norms:
\begin{lemma}\label{from unbounded to bounded}
Let $d:H_{1}\to H_{2}$ be a $\Gamma-$invariant closed densely defined operator between $\Gamma-$Hilbert modules (hence $\Dom(d)$ is $\Gamma-$invariant). Let $j_{i}:A_{i}\to H_{i}$ be bounded monomorphism of $\Gamma-$Hilbert modules for $i=1,2$. Assume that a bounded $\Gamma-$morphism $a: A_{1}\to A_{2}$ exists such that $j_{2}\, a=d\, j_{1}$, in particular $j_{1}$ maps to $\Dom(D)$.

If $d$ is $\Gamma-$Fredholm,  then $a$ is $\Gamma-$Fredholm. 
\end{lemma}
%
All the above properties are proved in \cite{Luc}, \cite{Shu}, \cite{ShuBook}, see also \cite{SchNachrichten}.

\section{The Kodaira laplacien on vector bundles of bounded curvature is \texorpdfstring{$\Gamma-$}{G-}Fredholm.}
\subsection{}\label{proof using bochner method}
Following arguments of Donelly-Li \cite{DonLi},  the Bochner method allows a comparison between the heat kernel associated to a Dirac operator and the Riemannian heat kernel. Numerous estimates are known for the latter kernel. This leads easily to a generalization of Theorem~\ref{cheeger-gromov} $(i)$ in the introduction:

\begin{theorem}[Atiyah, Cheeger-Gromov] Let $(X,g)$ be a Riemannian manifold of bounded curvature and positive injectivity radius. Let $\Gamma$ be a discrete group of isometries such that
\begin{align}
\Vol(X/\Gamma)<+\infty\,.
\end{align}
Let $(S,h,\nabla)\to X$ be a $\Gamma-$equivariant Clifford bundle with associated Dirac operator $D=c\circ \nabla$. 
Assume that the curvature of $h$ is bounded.
Then
\begin{align}
D: L^{2}(X,S)\to L^{2}(X,S)
\end{align}
is $\Gamma-$Fredholm.
\end{theorem}

\begin{proof}
For Dirac operators the Lichnerowicz-Weitzenböck Formula was derived in \cite[Thm.\ 3.52]{BerGetVer}, \cite[p.44]{Roe} (see also \cite[(6.1.22) p.430]{Wu}, and \cite[Chap.\ 10, Sec.\ 4]{Tay}): If $s$ is a section of the Clifford bundle $S$, then
\begin{align}
D^{2}s=\nabla^{*}\nabla s+ L.s\, .
\end{align}
Here $\nabla^{*}$ is the formal adjoint of the connection $\nabla: C^{\infty}(S)\to C^{\infty}(T^{*}M\otimes S)$, and $L$ is a smooth section of the Riemannian bundle $\Omega^{2}(X)\otimes End(S)$. The hypotheses made on the curvature of $h$ and $g$ imply that $L$ is a bounded section of this bundle.
%
%
%
%

\begin{lemma}Let $K$ and $H$ be the heat Kernels on $L^{2}(X)$ and $L^{2}(X,S)$ resp. Let -$b\in \RR$ be a lower bound of the curvature operator being part of the Lichnerowicz-Weitzenböck formula: for all smooth section with compact support $(Ls,s)\geq -b||s||^{2}$. Let $|H(t,x,y)|$ denote the norm of the morphism $H(t,x,y)\in E_{x}\otimes E^{*}_{y}$. Then
\begin{align}
\label{heat kernel domination}|H(t,x,y)|\leq e^{bt}K(t,x,y)\,.\end{align}

\end{lemma}
\begin{proof} The proof for the case of a compact manifold with boundary (and Dirichlet condition) is given in \cite[Thm.\ 4.3]{DonLi}. It is a variant of the Bochner method on estimating $\Delta ||s_{t}||^{2}$ for $s_{t}$ a section of the bundle $E\to [0,+\infty[\times M$ which solves the heat equation.
The arguments given in their work apply literally to a complete manifold.
\end{proof}
It is well known (see \cite{CheGroTay}, \cite{Gri},\cite[Thm.\ 4]{CheLiYau}, \cite[Cor.\ 3.1]{LiYau}) that the heat Kernel on a manifold with (sectional) curvature bounded by $k_{1}$ and injectivity radius bounded from below by $k_{2}>0$ is bounded by
\begin{align}\label{LiYau estimate}
K(t,x,y) \leq     ct^{-\frac{n}{2}}e^{-d^{2}(x,y)/at}   \quad  \text{ on }X\times X\times [0,T]\,.
\end{align}
for some suitable constants $a,c$ which depend only on $k_{1},k_{2}$ and $T$.
%
Estimates (\ref{heat kernel domination}), and (\ref{LiYau estimate}) imply that the heat kernel of $e^{-tD^{2}}$ is bounded.

But
\begin{align}
e^{-t D^{2}}=\int e^{-tx}dP_{D^{2}}(x)\geq \int 1_{[0,\epsilon]}(x).e^{-tx}dP_{D^{2}}(x)\geq e^{-t\epsilon}1_{[0,\epsilon]}(D^{2}) \,,
\end{align}
 and the trace is positive on positive operator, hence
 \begin{align}0\leq  \Tr_{\Gamma} 1_{[0,\epsilon]}(D^{2})\leq e^{t\epsilon}\Tr_{\Gamma}e^{-t D^{2}}<+\infty\,.\end{align}
\end{proof}
More than the $\Gamma-$Fredholm property was shown: For all $\lambda>0$, the spectral projection $1_{[0,\lambda]}(D^{2})$ have finite $\Gamma-$dimensional range.

\begin{lemma}
Let $A:H\to H$ be an unbounded, closed, and densely defined self-adjoint  operator such that for any bounded interval $I\subset \RR$, $1_{I}(A)$ is $\Gamma-$finite. Then for any bounded self-adjoint $\Gamma-$operator $B$, $A+B$ satisfies the same property.
\end{lemma}
\begin{proof}
The domain $\Dom(A+B)=\Dom(A)$ for $B$ is bounded. The operator $A+B$ is self-adjoint for $(A+B)^{*}=A^{*}+B^{*}=A+B$. Let $I\subset \RR$ be a bounded interval. Then $x\in \Ran(1_{I}(A+B))\Rightarrow x\in \Dom(A+B)$ and $||(A+B)x||\leq C||x||\Rightarrow ||A(x)||\leq C'||x||$. Therefore the spectral projection $1_{[-C',C']}(A)$ is injective on $\Ran(1_{I}(A+B))$ for an element in $\Ker(1_{[-C',C']}(A)\cap \Ran(1_{I}(A+B))$ belongs to $\Dom(A+B)=\Dom(A)$ and is equal to $1_{\RR\setminus [-C',C']}(x)$. Hence $||A(x)||>C'||x||$, if $x$ was non-vanishing (\cite[p.192,(2) in Prop.]{Weid}).
Hence  $\dim_{\Gamma}\Ran(1_{I}(A+B))\leq \dim_{\Gamma}\Ran(1_{[-C',C']}(A))$.
\end{proof}

\begin{corollary}
\begin{itemize}
\item[i)] In the above situation, let $a$ be a bounded section of the symetric endomorphisms bundle of $(S,h)\to M$. Then $D':=D+c(a+a^{*})$ is also $\Gamma-$Fredholm.
\item[ii)] Let $(E,h,\nabla_{h},\nabla)\to (X,g)$ be a $\Gamma-$equivariant vector bundle equipped with a Riemannian connection $\nabla_{h}$ and another connection $\nabla$. Assume that the curvature of $\nabla_{h}$ is bounded and that $\nabla_{h}-\nabla$ is bounded. Then the generalized Dirac operator $\nabla+\nabla^{*}:\Lambda^{.}T^{*}(X)\otimes E\to \Lambda^{.}T^{*}(X)\otimes E$ is $\Gamma-$Fredholm.
\item[iii)] Assume moreover $(X,g)$ is Hermitian of bounded curvature and that $E$ in $(ii)$ is moreover holomorphic with Chern connection $\nabla_{h}$. Then
$\nabla^{1,0}+{\nabla^{1,0}}^{*}$ and $\nabla_{h}^{0,1}+{\nabla_{h}^{0,1}}^{*}$ are $\Gamma-$Fredholm.
\end{itemize}
\end{corollary}


\section{application to a Galois \texorpdfstring{$\partial\dbar-$}{ddbar-}lemma}
\begin{corollary}\label{elliptic lifting} Let $X$ be a complete manifold of bounded curvature and positive injectivity radius. Let $\Gamma$ be a discrete group of isometries such that $X/\Gamma$ has finite volume. Let $(S,D,h)\to X$ be a $\Gamma-$equivariant Dirac bundle such that the curvature of $h$ is bounded.
Let $\alpha\in L^{2}(X,S)$ which belongs to $\Adh{\Ran(S)}$. Then there exists an $r\in VN(\Gamma)$ injective with dense range, such that $r.\alpha\in \Ran(S)$.
\end{corollary}
\begin{proof} Apply \cite[Lemma 2.15]{Din2013} to the bounded operator $D^{2}(1+D^{2})^{-1}: L^{2}(X,S)\to L^{2}(X,S)$, which is $\Gamma-$Fredholm with the same range than $D^{2}$, because $(1+D^{2})^{-1}$ is an isomorphism from $L^{2}(X,S)$ to $\Dom(D^{2})$.
\end{proof}

\begin{corollary}[A Galois $\partial\dbar-$lemma.]\label{Galois ddbar lemma} Let $p:X\to X/\Gamma$ be a Galois covering of a complete Kähler manifold with Galois group $\Gamma$. Assume $\Vol(X/\Gamma)<+\infty$, and assume that $X$ is of bounded curvature and positive injectivity radius. 
Let $x$ be a square integrable closed $(p,q)-$form in $X$, which is orthogonal to the space of square integrable harmonic forms.

Then there exists $r\in \Vn(\Gamma)$ injective, and a square integrable form  $y$ of type $(p-1,q-1)$, which is in the domain of $\partial\dbar$ such that  $$r.x=\partial\dbar y\,.$$
\end{corollary}
\begin{proof} It is similar to the proof given in \cite{Din2013}. We give it for the convenience of the reader.  By Hodge decomposition $L^{2}(X, \Lambda T^{*}X)=\H_{(2)}(X)\overset{\perp}{\oplus} \Adh{\Ran(\Delta(1+\Delta)^{-1})}$ holds. The preceding corollary implies that there exists $r\in \VN(\Gamma)$ injective with dense range such that $r.x=\Delta y$. As $X$ is Kähler,
\begin{eqnarray} r.x=\Delta y=\Delta_{\dbar} y=\Delta_{\partial} y & = & (\dbar+\dbar^{*})^{2}y=(\partial+\partial^{*})^{2} y\,.
\end{eqnarray}
Therefore, we may assume that $y$ is of pure type $(p,q)$.
The completeness of the metric implies that $y$ belongs to the domain of $\dbar+\dbar^{*}$ and of $\partial+\partial^{*}$ (see \cite[Cor.\ 6]{AndVes}) and $\Delta$, hence to the domain of $\dbar \dbar^{*},\dbar^{*}\dbar,\ldots$

The fact that $r.x$ is a $d-$closed $(p,q)-$form implies that it is both $\partial$ and $\dbar-$closed. Hence $$r.x-\dbar \dbar^{*} y=\dbar^{*}\dbar y \in \Ker(\dbar)\cap \Ker(\dbar)^{\perp}\,. \text{ This form is vanishing.}$$
Now $\dbar^{*}\dbar y=0$ implies that $\dbar y=0$. Also the equation $$r.x-\partial \partial^{*} y=\partial^{*}\partial y \in \Ker(\partial)\cap \Ker(\partial)^{\perp}$$ implies that $\partial y=0$. Hence  $\partial y=\dbar y=0$ and $$r.x=\dbar\dbar^{*}y=-i\dbar (\Lambda \partial-\partial \Lambda )y=-i\partial \dbar \Lambda y\,.$$

\end{proof}
\begin{remark}
\begin{itemize}
\item[1)]
The more traditional proof using first a $\dbar-$primitive, then a $\dbar^{*}-$primitive was given in \cite{Din2013,Din2018}. However, it requires uniform Sobolev spaces to justify integrations by parts. Here, using the elliptic second order operator $\Delta$, one obtains a form which belongs already to the Sobolev space of order two.

\item[2)] Note that the form $y$ from the proof, which satisfies $r.x=\partial\dbar \Lambda y$, is also a closed form of type $(p,q)$. Hence, one may iterate the procedure.
\item[3)] The above results are valid in any of the uniform Sobolev spaces, that is for any current of the form $\Delta^{k}[\alpha]$, where $\alpha$ is square integrable (see \cite{Din2018}).
\end{itemize}
\end{remark}

A version of the Galois $\partial\dbar-$lemma for unitary bundles with Higgs fields, and for harmonic Higgs bundles will be given in Section~\ref{ddbar lemma for higgs bundles}.

%
\section{Computation of the $l^{2}-$invariants}
The heat equation proof of the index theorem needs bounds on the remainder term in the asymptotic expansion of the heat kernel $e^{-tD^{2}}$ as $t$ goes to zero. Hence, in this section we restrict ourselves to bundles of bounded geometry over manifolds of bounded geometry. Note however that Riemannian metric of bounded curvature and positive injectivity radius can be uniformly approximated by a metric of bounded geometry (see \ref{approximation by bounded geometry}).

\subsection{Invariance of the super-trace}
\subsubsection{Definition of the supertrace}
Let $E=E_{+}\oplus E_{-}$ be a $\ZZ_{2}-$graded vector space. Let $\eta$ be the grading operator ($\eta_{|E^{\pm}}=\pm Id_{E^{\pm}}$). Then $\End(E)$ decomposes into odd and even operators:
\begin{align}
\End^{+}(E) &=\Hom(E_{+},E_{+})\oplus \Hom(E_{-},E_{-})\\
\End^{-}(E) &=\Hom(E_{+},E_{-})\oplus \Hom(E_{-},E_{+})\\
A\in \End^{\pm}(E)&\iff A=\pm\eta\circ A\circ\eta
\end{align}

The supercommutator is defined as
$ab-(-1)^{|a| |b|}ba$ on $\ZZ_{2}-$homogeneous morphisms and extended by linearity.
Let $\Tr$ be a trace defined on a $\eta-$stable sub-algebra of $\End(E)$. The supertrace is given by \begin{align}\Tr_{s}(A):=\Tr(\eta\circ A)\,.\end{align}
Then $\Tr_{s}(A)=\Tr_{E^{+}}(A_{|E^{+}})-\Tr_{E^{-}}(A_{|E^{-}})$ on even operators, whereas $\Tr_{s}(A)=0$ on odd operators, and the super trace vanishes on a supercommutator $\Tr_{s}[A,B]_{s}=0$ (see \cite[Prop.\ 1.31]{BerGetVer}).

This formalism is applied to a smoothing operator $f(D)$ on $l^{2}(X,S)$ according to Theorem~\ref{functional calculus}, where $(S,h,D)\to (X,g)$ is a $\Gamma-$equivariant $\ZZ_{2}-$graded Dirac operator. Let $k_{f}(x,x)\in \C^{\infty}(\End(S))$ be the restriction of its Schwartz kernel along the diagonal. The pointwise decomposition $S_{x}=S^{+}_{x}\oplus S^{-}_{x}$ defines $\Tr_{s}k_{f}(x,x)$ as the pointwise super trace of this endomorphism. $\Gamma-$equivariance implies that $x\mapsto \Tr_{s}k_{f}(x,x)$ defines a function $y\mapsto \Tr_{s}(k_{f})(y)$ on $Y=X/\Gamma$. Let $F\subset X$ be a fundamental domain for the $\Gamma-$action. Results from Section~\ref{Gamma-dimension} imply
\begin{align} \Tr_{\Gamma,\, s}f(D)=\int_{F} \Tr_{s}k_{f}(x,x)dV_{g}(x)=\int_{y\in X/\Gamma} (\Tr_{s}k_{f})(y)dV_{g}(y)\,.
\end{align}

\subsubsection{Invariance of the super trace}

The following lemma, based on the existence of a functional calculus $f\mapsto f(D)$ within the $\Gamma-$trace class operators, proves that the $\Gamma-$super trace of the heat operator is time independent (see e.g. \cite[Prop.\ 11]{Roe-elliptic}, and \cite{BerGetVer,CheGro85-characteristic}):
\begin{lemma} Let $(S,h,D)\to (X,g)$ be a $\Gamma-$equivariant $\ZZ_{2}-$graded  Dirac bundle of bounded geometry.
Let $f_{1}$ be a rapidly decreasing smooth function on $\RR_{+}$ (in the Schwartz class) such that $f_{1}(0)=0$ then
\begin{align}\Tr_{\Gamma,s}f_{1}(D^{2})=0\,.
\end{align}
In particular $\Tr_{\Gamma,s}f(D^{2})$ is independent of the function $f$ in the Schwartz class such that $f(0)=1$.
\end{lemma}
\begin{proof} Since $S\to X$ is a Dirac bundle of bounded geometry,  \cite[Thm.~5.5]{Roe} implies that the functional calculus $f\mapsto f(D):=\int_{\RR}\hat f(t)e^{itD}dt$, from the Schwartz class to bounded operators on $L^{2}(X,S)$, gives uniform operators (loc.cit.\ Def.~5.3 and Prop.~5.4). Hence the arguments of \cite[Chap.~15]{Roe-elliptic} apply:
There exist functions $h_{0},h_{1}, h_{2}$ in the Schwartz class such that $f_{1}(x)=xh_{0}(x)$ and $h_{0}(x^{2})=h_{1}(x)h_{2}(x)$. Then $f_{1}(D^{2})=D^{2}h_{1}(D)h_{2}(D)=\frac{1}{2}[Dh_{1}(D),Dh_{2}(D)]_{s}$. One concludes the proof using the vanishing of the super trace on super commutators.
 \end{proof}

\subsection{Convergence  of the Von Neumann index of $\mathbf D$ to the characteristic integral}
We recall the fundamental identity between $l^{2}-$characteristic on $X$ and the integral of Chern Gauss Bonnet on the manifold $Y=X/\Gamma$ (for a quick review on the Murray-Von Neumann dimension, we refer to \cite[Sec.~6]{CheGro85-characteristic}, and \cite{CheGro85-bounds}).
\begin{theorem}[see Atiyah\cite{Ati}, Cheeger-Gromov \cite{CheGro85-characteristic}]\label{index theorem} Let $\Gamma$ be a discrete group of isometries of a Riemannian manifold of bounded geometry  $(X,g)$ such that $\mathrm{Vol}(X/\Gamma)<+\infty$.
Let $(D,S)\to (X,g)$ be a $\Gamma-$equivariant $\ZZ_{2}-$graded Dirac operator of bounded geometry\footnote{see also Sec.~\ref{invariance by change of metric}}.
Let $P$ be the bounded endomorphism of $l^{2}(X,S)$ defined by the orthogonal projection onto $\Ker(D)$.
Then X
\begin{align}
\chi_{(2)}(D,S):=\Tr_{\Gamma,s}(P)=\int_{F}\Tr_{s}(P)dv= \int_{Y}\hat{A}(TY)\wedge \Ch(S/\Delta)\,.
\end{align}
\end{theorem}
$\hat A(X)$ is by definition the $\hat A-$genus of the manifold and $\Chern(S/\Delta)$ is the relative Chern character (see \cite[4.25]{Roe} for a definition).

\subsubsection{} We sketch the proof of the theorem, since it is well referenced (see e.g.~\cite{Roe,MaMar} for the cocompact case).

\begin{lemma}[{see e.g.\cite[Part II]{Roe}}]  As $t\mapsto +\infty$, the Schwartz kernel of $e^{-tD^{2}}$ tends to the Schwartz kernel of $P$ (the projection onto the kernel of  $D$) in the Fr\'{e}chet topology of $C^{\infty}(S\boxtimes S^{*})$ (which is the topology of uniform convergence on compact subsets of $X\times X$). Consequently the Schwartz kernel of $P$ is uniformly bounded together with all its covariant derivatives.
\end{lemma}
%
%
Then one uses
 (see \cite[Prop.~2.11]{Roe}; see also \cite[Thm.~7.15, p.101]{Roe-elliptic}, or  \cite[Chap.~4]{BerGetVer}):
\begin{proposition} Let $(S,D)\to (X,g)$ be a Dirac operator of bounded geometry on a manifold of bounded geometry with $X$ oriented. Then the operator $e^{-tD^{2}}$ is represented by a uniformly bounded smoothing kernel $k_{t}(x,y)$, and there is an asymptotic expansion
\begin{align*}
k_{t}(x,x)\sim (4\pi t)^{-\frac{n}{2}} \sum_{k\geq 0} t^{k}\Psi_{k}(x)\,,
\end{align*}
where the $\Psi_{k}$ are smooth sections of $\End(S)\otimes \Lambda^{n}T^{*}X$, locally computable in terms of the curvatures of $X$ and $S$ and their covariant derivatives. Moreover the remainder terms, which appear implicitly in the asymptotic expansion, are uniformly bounded in $x\in X$.
\end{proposition}

Now assume that $(S,D)\to (X,g)$ is $\Gamma-$equivariant, $\Vol(X/\Gamma)<+\infty$. Then 
the terms $\Psi_{k}(x)$, which are $\Gamma-$equivariant, are integrable on $X/\Gamma$. Recall also that the remainder terms in the asymptotic expansion are bounded. Hence,  one has the asymptotic expansion
%
\begin{eqnarray}
\int_{X/\Gamma}\Tr_{s}(k_{t})(y)dV_{g}\sim \frac{1}{ (4\pi t)^{\frac{n}{2}}}\left(\int_{X/\Gamma}\Tr_{s}\Psi_{0}dV_{g}+t \int_{X/\Gamma}\Tr_{s}\Psi_{1}dV_{g}+\ldots\right)
\end{eqnarray}
(see p.147 of Roe's book).
Observe that the left hand side is constant in $t$, hence in the right hand side, the terms vanish if $k<\frac{n}{2}$ (even if $k\not=\frac{n}{2}$ see \cite{BerGetVer} p.141) and
\begin{eqnarray}
\int_{X/\Gamma}\Tr_{s}(e^{-tD^{2}} )&= &   \int_{X/\Gamma} \frac{1}{ (4\pi )^{\frac{n}{2}}} \Tr_{s}\Psi_{\frac{n}{2}}dV\,.
\end{eqnarray}

In particular the index is vanishes, if $n$ is odd.

It is known (\cite[Chap.~4]{BerGetVer}, \cite[Chap.~12]{Roe}) that
 $\frac{1}{ (4\pi )^{\frac{n}{2}}} \Tr_{s}\Psi_{\frac{n}{2}}dV$
is the component in degree $n$ of \break$\hat A(Y)\Chern(S/\Delta)$,
where $\hat A(Y)={\det}^{\frac{1}{2}}\left( \frac{ (\frac{i}{2\pi}R)/2 }{ \sinh(\frac{i}{2\pi}R)/2) } \right)$ is the $\hat A-$genus of the manifold (where $R$ denotes the curvature endomorphism) and $\Chern(S/\Delta)=\Tr^{S/\Delta}(\exp(-\frac{1}{2i\pi}F^{s}))$ is the relative Chern character (see  \cite[Example~2.28 and 4.25]{Roe} or \cite[Sec.~1.5, and Sec.~4.1]{BerGetVer}). This completes the proof of the theorem.


%

%
\begin{example}
 Let $(E,h)\to (Y,\omega)$ be a holomorphic bundle over a Kähler manifold. Then $\Lambda (T^{*(0,1)}(Y))\otimes E\to Y$ is a Clifford  bundle with associated Dirac operator $\sqrt2(\dbar+\dbar^{*})$. Then using \cite[p.148]{BerGetVer}, one sees that
\begin{align}
\chi_{(2)}(X,E,h,\dbar):=\sum_{i}(-1)^{i}\dim_{\Gamma}\H^{0,i}_{\dbar(2)}(X,E,h)=\int_{Y} \Todd(Y,\omega)\,\Chern(E,h)
\end{align}
where $\Todd(Y,\omega)=\det\left( \frac{ \frac{i}{2\pi} R^{+} }{e^{\frac{i}{2\pi} R^{+}}-1}\right)$ and $R^{+}$ is the curvature induced by the Kähler metric on the holomorphic tangent bundle $T^{(1,0)}(Y)$ and $\Chern(E,h)=\Tr\exp\left(\frac{i}{2\pi}\Theta(h)\right)$ is the Chern character of the hermitian bundle $(E,h)$.\footnote{Recall that $\Todd(E)=1+c_{1}/2+(c_{1}^{2}+c_{2})/12+\ldots$ is multiplicative: $\Todd(E\oplus E')=\Todd(E)\cdot\Todd(E')$, and $\Ch(E)=\Rank(E)+c_{1}+(c_{1}^{2}-2c_{2})/2+\ldots$ is a homomorphism from the ring of bundles with connections to the ring of differential forms.
}
\end{example}

\subsection{Interpretation of the characteristic integrals}\label{Interpretation of the characteristic integrals}
When $Y=\bar Y\setminus D$ is the complement of a normal crossings divisor $D$ in a compact complex manifold $\bar Y$ and $E\to Y$ is the restriction of a holomorphic bundle $\bar E\to \bar Y$, the characteristic integrals associated to $(E,h)\to (Y,\omega_{\bar Y,D})$ will be related to characteristic integrals of $\bar E$ over the logarithmic pair $(\bar Y,D)$.

In \cite{Mum} Mumford introduced the notion of good metric on $\bar E_{|\bar Y\setminus D}\to \bar Y\setminus D:=Y$.
In the same article, he proved that automorphic vector bundles admit such compactifications (see also \cite{Zuc,HarPho}).
Other boundary behaviors  along the normal crossings divisor $\bar Y\setminus Y$ are studied in \cite{BurKraKuh}.

A smooth form on $\Delta^{p}\setminus D$ has {\em Poincaré growth}, if it is bounded in the Poincaré metric: \begin{align*}\exists C\geq 0\quad \text{s.t. }\quad |\eta(t_{1},\ldots,t_{p})|^{2}\leq C_{p}\omega_{P}(t_{1},t_{1})\ldots\omega_{P}(t_{p},t_{p})\,.\end{align*}
It defines an integrable current ($\int_{\Delta^{p}}|\eta|<+\infty$).
A form is good if both $\eta$ and $d\eta$ have Poincaré growth. Hence $d[\eta]=[d\eta]$, in particular $\eta$ has vanishing residues.
\begin{definition}A smooth hermitian metric $h$ on $E$ is good on $\bar Y$, if for all $y\in \bar Y\setminus Y$ and all bases $e_{1},\ldots,e_{k}$ of $\bar E$ in a neighborhood $\Delta^{n}$ near $x$ such that $\bar Y\setminus Y$ is defined by $z_{1}\cdot\ldots\cdot z_{p}=0$. Let $h_{i,j}=h(e_{i},e_{j})$. Then
\begin{itemize}
\item[i)] $|h_{ij}|, \quad (\det h)^{-1}\leq C(\sum_{i=1}^{p}\log|z_{i}|)^{2n'}$.
\item[ii)] The 1-form  $(\partial h.h^{-1})_{ij}$ is good on $\bar Y\cap U$.
\end{itemize}
\end{definition}
\begin{example}
Let $D=\bar Y\setminus Y$ be a normal crossings divisor. Let $\omega$ be the Kähler metric on $Y$ of Poincaré type along $D$ as in \ref{Poincare metric}. Then $\omega$ is good on the bundle of logarithmic forms $\bar E=\Omega^{1}_{\bar Y}(\log(D))\to \bar Y$, which is the extension of $(\Omega_{Y},\omega)\to Y$. It is enough to prove this in the one dimensional case: The square norm in the Poincaré metric of the local frame $\frac{dz}{z}$ of $\Omega^{1}([0])$ is $(\log|z|^{2})^{2}$, and $\partial h. h^{-1}= \frac{2}{-\log|z|^2}\frac{dz}{z}$
is bounded with respect to the Poincaré metric.
  \end{example}

We refer to the original papers \cite{CorGri,Mum} for a proof of the following theorem:

\begin{theorem}[Cornalba-Griffiths, Mumford]
\begin{itemize}
\item[i)] Let $(E,h)\to Y$ be a Hermitian vector bundle, whose Chern curvature $\Theta(h)$ is bounded in the Poincaré metric near the normal crossings divisor $\bar Y\setminus Y$. Then there exists an algebraic vector bundle $\bar E\to \bar Y'$ on a compactification of $Y$ dominanting $\bar Y$ such that $\bar E_{|Y}=E$.
\item[ii)] If $h$ is a good metric, then for any standard coordinate neigborhood $(U,(z_{i}))$ as above,
    \begin{align}
    \Gamma(\bar U,\bar E)=\{s\in \Gamma(\bar U \setminus D,E) \text{ s.t. }\exists\, C,n',\, h(s,s)\leq C(\sum_{i=1}^{p}\log|z_{i}|)^{2n'}\}\,.
    \end{align}
Hence there exists at most one extension of $E$ to a vector bundle $\bar E$ such that $h$ is good on $\bar E$.
\item[iii)]  If $h$ is good then $c_{k}(E,h)$ is good and the currents $[c_{k}(E,h)]$ represent the cohomology class $[c_{k}(\bar E)]\in H^{2k}(\bar Y,\bar E)$.
\end{itemize}
\end{theorem}
%
%
The characteristic integrals given by Atiyah-Cheeger-Gromov's theorem have a good interpretation for a good metric. The next formula applies in particular for holomorphic tensor bundles equipped with the metric induced by the Kähler metric with Poincaré growth along $D$. Let $T\bar Y(-\log(D))$ be the logarithmic tangent bundle, which is by definition the dual of the bundle of logarithmic forms $\Omega^{1}_{\bar Y}(\log(D))\to \bar Y$.

\begin{corollary}\label{characteristic logarithmic}
Let $p:X\to Y$ be a Poincaré covering and $(E,h)\to Y$ be a hermitian bundle. Assume that its pullback $(E,h)\to (X,p^{*}(\omega_{\bar Y,D}))$ is of bounded geometry.
\begin{itemize}
\item[i)] Let $h_{\log(D)}$ be a Hermitian metric on $T\bar Y(-\log(D))$. Then
\begin{align}\chi_{(2)}(X,E,h,\dbar)=\int_{Y}\Todd(T\bar Y(-\log(D)), h_{\log(D)})\Ch(E,h)\,.
\end{align}
\item[ii)] Assume $h$ is a good metric for $\bar E\to \bar Y$.
Then
\begin{align}
\chi_{(2)}(X,E,h,\dbar )=\chi((\bar Y,D), \bar E):=\int_{\bar Y}\Todd(T\bar Y(-\log(D)))\Ch(\bar E)\,.
\end{align}
\end{itemize}
\end{corollary}
\begin{proof}
\begin{itemize}
\item[i)]
Since the metric $\omega_{\bar Y,D}$ is good for $T\bar Y(-\log(D))$, there exists a transgression formula
\begin{align}
\Todd(T\bar Y(-\log(D)))_{|Y},\omega_{\bar Y,D})=\Todd(T\bar Y(-\log(D)),h_{\log(D)})+dT
\end{align}
such that $T$ is bounded in the Poincaré metric (see the argument in \cite[Thm.~1.4 ]{Mum}). The form $T\wedge \Ch(E_{0},h_{0})$ is bounded in the Poincaré metric, it is therefore integrable. Gaffney's theorem applied to the complete manifold $(Y,\omega_{\bar Y,D})$ implies that $\int_{Y}dT\wedge \Ch(E_{0},h_{0})=0$ (see \cite{KobRyo}).
\item[ii)] As above, one has $\Ch(E,h)=\Ch(\bar E,\bar h)_{|Y}+dT'$ with $T'$ bounded in the Poincaré metric, hence it is possible to integrate by parts one more time in the above formula. Finally one uses the fact that characteristic integrals over compact manifolds do not depend on the metric.
\end{itemize}
\end{proof}

\begin{example}
\begin{itemize}
\item[i)] Mumford proved in \cite{Mum} that an automorphic vector bundle on a quotient of a symmetric domain by a net lattice admits an extension such that the quotient metric is good for the extension. Tensor bundles extend via tensor products of the bundle of logarithmic forms or logarithmic vector fields.
\item[ii)]
In \cite{LiuSunYau2004,LiuSunYau2005}, Liu, Sun and Yau introduced the Ricci and the perturbed Ricci metrics on $\mathcal{M}_{g}$, the moduli space of Hyperbolic Riemann surfaces. These metrics are equivalent to the Bergmann metric and to the Kähler-Einstein metric on $\mathcal{M}_{g}$. They prove that these metrics are of bounded geometry on the Teichmüller space (see also \cite{GSch}). One will consider theses metrics only on the finite branched covers of $\M_{g}$ (e.g.\ moduli of Riemann surfaces with level structure), which are manifolds.

In \cite{LiuSunYau2014}, the authors prove that the Weil-Petersson metric and the above two metrics induce good Hermitian metrics on the logarithmic tangent bundle of the Deligne-Mumford compactification $(\overline{ \mathcal{M}_{g}},\overline{\mathcal{M}_{g}}\setminus \mathcal{M}_{g})$.
\end{itemize}
\end{example}

\begin{example}
\begin{itemize}
\item[i)] For any holomorphic vector bundle $E$, let $\Lambda^{0}E=\CC$ be the trivial holomorphic bundle. For any $i> 0$, the Poincaré type metric $\omega_{\bar Y,D}$ is good for $\overline{\Lambda^{i}\Omega^{1}}:=\Lambda^{i}\Omega^{1}_{\bar Y}(\log(D))=\Omega^{i}_{\bar Y}(\log(D))$ and, for $i\geq 0$,
\begin{align}\chi_{(2)}(X,\Omega^{i}_{X}, \dbar, p^{*}(\omega_{\bar Y,D}))=\int_{\bar Y}\Todd (T\bar Y(-\log(D)))\Chern(\Omega^{i}_{\bar Y}(\log(D)))\,.\end{align}
In particular $\quad\chi_{(2)}(X,\CC,\dbar,p^{*}(\omega_{\bar Y,D}))=\int_{\bar Y}\Todd (T\bar Y(-\log(D)))$.
\item[ii)] For any holomorphic bundle $\bar E\to \bar Y$ of rank $r$, we have (see \cite[Ex.\ 3.2.5]{Ful})
\begin{align}\label{relation for Todd}
\sum_{i\geq 0}(-1)^{i}\Chern(\Lambda^{i}\bar E^{*})=c_{r}(\bar E)\Todd(\bar E)^{-1}\,.
\end{align}
In particular, if $rank (\bar E)=n=\dim\bar Y$, then (see \cite[Ex.\ 183.7]{Ful})
\begin{align}\label{relation for Todd 2}
\Todd(T\bar Y(-\log(D))) c_{n}(\bar E)\Todd(\bar E)^{-1}=c_{n}(\bar E)=\Todd(T\bar Y) c_{n}(\bar E)\Todd(\bar E)^{-1}\,.
\end{align}
\end{itemize}
 Therefore:
\begin{corollary} Let $\bar Y$ be a compact Kähler manifold, $D$ a normal crossings divisor, $Y=\bar Y\setminus D$, and let $p:X\to Y$ be a Poincar\'{e} covering. Let $\bar E\to \bar Y$ be  a vector bundle of rank $n=\dim X$ and let $h$ be a metric on $E:=\bar E_{|Y}$, which is good for $\bar E$. Then
\begin{itemize}
\item[i)]
\begin{align}
\sum_{i}(-1)^{i}\chi_{(2)}(X,\Lambda^{i}E^{*},\dbar)=\int_{\bar Y}c_{n}(\bar E)=\sum_{i}(-1)^{i}\chi_{}(\bar Y,\Lambda^{i}\bar E^{*})\,.
\end{align}

\item[ii)] In particular, for the trivial real bundle $\CC\to (X,p^{*}(\omega_{\bar Y,D}))$, one gets

\begin{align}
\chi_{(2)}(X,\CC,d):=\sum_{i}\dim_{\Gamma}\H^{i}_{d(2)}(X)=\int_{\bar Y}c_{n}(T\bar Y(-\log(D)))=\chi(\bar Y-D)\,.
\end{align}
\end{itemize}
\end{corollary}
\begin{proof}
\begin{itemize}
\item[i)]
If $h$ is good for $\bar E\to \bar Y$, it induces a good metric on $\Lambda^{i}\bar E$. From (\ref{relation for Todd 2}), the $l^{2}-$Index Theorem~\ref{index theorem} and Grothendieck-Riemann-Roch Theorem, one gets
\begin{align}
\sum_{i}(-1)^{i}\chi_{(2)}(X,\Lambda^{i}E^{*},\dbar )=\int_{\bar Y}c_{n}(\bar E)=\sum_{i}(-1)^{i}\chi_{}(\bar Y,\Lambda^{i}\bar E^{*})\,.
\end{align}
\item[ii)] The metric $p^{*}(\omega_{\bar Y,D})$ is Kähler, hence Hodge decomposition reads
\begin{align}\H^{i}_{d (2)}(X,p^{*}(\omega_{\bar Y,D})) &=\oplus_{k+l=i}\H^{k,l}_{\dbar (2)}(X,p^{*}(\omega_{\bar Y,D}))\,.
\end{align}
Hence $(i)$ implies
\begin{align}
\sum_{i}(-1)^{i}\chi_{(2)}(X,\Lambda^{i}\Omega^{1}_{X},\dbar)=\sum_{i}(-1)^{i}\chi(\bar Y,\Lambda^{i}\Omega^{1}_{\bar Y}(\log D))
=\int_{\bar Y}c_{n}(T\bar Y(-\log(D)))\,.
\end{align}

Deligne's Mixed Hodge theory (see \cite{Deldeux} (3.2.2)) implies $$H^{i}(\bar Y\setminus D,\CC)=\HH^{i}(\bar Y,(\dlog{.}{D},d)).$$ This gives the result.
\end{itemize}
\end{proof}

\begin{remark}Homotopy invariance will reprove that the $l^{2}-$Euler characteristic  is equal to the Euler characteristic of $Y$. \end{remark}

\end{example}
\subsection{The case of a curve.}
\begin{example}
Let $(\bar Y,D)$ be a punctured Riemann surface with $Y=\bar Y \setminus D$ hyperbolic. Denote by $p:X=D(0,1)\to Y$ a uniformization map with fundamental domain $F$ for $\Gal(p)$. The Poincaré metric $\omega_{P}$ on $D(0,1)$ descends to a complete hyperbolic metric $\omega_{\bar Y,D}$ on $Y$, of bounded volume and curvature -$1$. For this metric on $D(0,1)$, one has
\begin{align*}
\H^{0}_{(2)}(D(0,1))=0 \,;& & \H^{1,0}_{(2)}(D(0,1))\sim\overline{\H^{0,1}_{(2)}(D(0,1))} \,; &&\H^{2}_{(2)}(D(0,1))=0\,.
\end{align*}
The Bergmann form $B(z)dV$ of the space $\H^{1,0}_{(2)}(D(0,1))$ is independant of the metric, hence
 \begin{align}
B(z)dV&=\sum_{k=0}^{+\infty}\frac{k+1}{2\pi}|z|^{2k}idz\wedge d\bar z=\frac{1}{2\pi}\frac{idz\wedge d\bar z}{(1-|z|^{2})^{2}}=\frac{1}{4\pi}\omega_{P}\,.\\
\dim_{\Gal(p)}\H^{1,0}_{(2)}(D(0,1))&=\int_{F}\frac{1}{4\pi}\omega_{P}=\frac{1}{4\pi}\Vol_{\omega_{\bar Y,D}}(Y)\,.
\end{align}
Assume $\bar Y$ has genus $g$ and $D$ contains $s$ points. The Gauss-Bonnet formula for a hyperbolic Riemann surface of finite volume (\cite[Cor.\ 10.4.4]{Bea}) states that $\Vol_{\omega_{\bar Y,D}}(Y)=2\pi(2g-2+s)=-2\pi\chi(Y)$. Hence
\begin{align}
\chi_{(2)}(X,\CC,d) =2\,\chi_{(2)}(X,\CC,\dbar)=-\frac{1}{2\pi}\Vol_{\omega_{\bar Y,D}}(Y)=2-2g-s=2\int_{\bar Y}\Todd(K^{*}\otimes[-D])\,.
\end{align}
The last equality is true, since $\Todd(K^{*}\otimes[-D])=(1+\frac{1}{2}c_{1}(K^{*}\otimes [-D])=1-\frac{1}{2}c_{1}(K)-\frac{1}{2}c_{1}(D)$.
In particular $\chi_{(2)}(X,\CC,\dbar)$ belongs to $\frac{1}{2}\ZZ$ in general. The above corollary reproves the Gauss-Bonnet Theorem. The equality between $l^{2}-$characteristic and Euler characteristic is also a consequence of the homotopy invariance of the $l^{2}-$Betti numbers (Sec.~\ref{invariance of the Betti numbers}) because $Y$ possesses a retract to a wedge of $2g+s-1$ circles.

As an example, let $\bar Y$ be an elliptic curve and $D=\{P\}$ be a point in $\bar Y$. Then $Y=\bar Y\setminus\{P\}$ is a hyperbolic Riemann surface with fundamental group $\FF_{2}$. A fundamental domain $F\subset \Adh{D(0,1)}$ for the Galois action is bounded by two pairs of geodesics $((z, \gamma_{1}(z)),(\gamma_{2}(z),\gamma_{2}\gamma_{1}(z)))$ $((z,\gamma_{2}(z)),(\gamma_{1}(z),\gamma_{1}\gamma_{2}(z)) )$ (where a geodesic $(z,w)$ joins the points  $z,w\in \partial D(0,1)$).

Hence $F$ is a hyperbolic square with vertex on the unit circle. It is the disjoint union of two ideal triangles (cut along the geodesic $(z,\gamma_{2}\gamma_{1}(z))$) each of which has area equal to $\pi$ in the Poincaré metric so that $\Vol(F)=2\pi$; $\chi(Y)=-1$ for $Y$ retracts to a wedge of two circles.

\end{example}

\begin{example}$\CC^{*}\subset \PP^{1}$. The covering $p:\CC\to \CC^{*}$ $z\mapsto e^{iz}$ is a Poincaré covering. It is easily seen that the $l^{2}-$reduced cohomology of $(\CC,p^{*}(\omega_{\PP^{1},[0]+[\infty]}))$ is vanishing as $c_{1}(T_{\PP^{1}}(-2))$.
\end{example}
We now refer to Example~\ref{example} below.
\begin{example}
Let $p:X\to Y$ be the universal covering map. Let $L\oplus L^{-1}\to Y$ be the Hitchin bundle. Then $\chi_{(2)}(X,Sym^{n}(L\oplus L^{-1}),\dbar)=-\chi_{(2)}(X,L^{n+2},\dbar)$ is linear in $n$ as predicted by the $l^{2}-$Riemann-Roch theorem. Serre duality gives that $b^{0}_{(2)}(X,L)=b^{1}_{(2)}(X,L)$ (hence $\chi_{(2)}(X,L,\dbar)=0$) but $b^{0}_{(2)}(X,L)$ vanishes, because there are no square integrable half-forms with respect to the Poincaré metric on the unit disc.
\end{example}

 Let  $(E,h)\to Y$ be a hermitian holomorphic bundle over a punctured curve $Y=\bar Y\setminus\{P_{1},\ldots,P_{s}\}$ such that its curvature is bounded in the Poincaré metric. According to \cite[Lemma~6.1, p.749]{Sim90} one has $$\int_{Y}C_{1}(E,h)=\deg(E,\{E_{\alpha,P_{i}}\})$$
where $\deg(E,\{E_{\alpha,P_{i}}\})$ is an algebraic degree defined in terms of the prolongation bundle (loc.cit.\ Section~3 and Prop.~3.1).
One obtains the following version of Atiyah's Riemann-Roch theorem:
\begin{theorem}
Let $(E,h)\to Y$ be a hermitian holomorphic bundle over the punctured curve $Y$ with bounded curvature together with its derivatives. Let $p:X\to Y$ be a Galois Poincaré covering with covering group $\Gamma$ (iff the class of circles around punctures are not of finite order in $\pi_{1}(Y)/p_{*}(\pi_{1}(X))$). Then, with $r=Rank(E)$,
\begin{align}
\chi_{(2)}(X,E,h,\dbar)&=r[1-g-\frac{s}{2}]+\deg(E,\{E_{\alpha,P_{i}}\}) ,\\
\chi_{(2)}(X,(E,h)\otimes \Omega^{1},\dbar)&=r[g-1+\frac{s}{2}]+\deg(E,\{E_{\alpha,P_{i}}\})\,. 
\end{align}

\end{theorem}

\subsubsection{Anticipating the harmonic bundle case}
Characteristic integrals are more delicate to compute in the higher dimensional case (see e.g.\ \cite{BurKraKuh} for this problem).
For (tame) harmonic bundles, which will be defined in the next section, stronger statements can be made.

\begin{theorem}\label{characteristic integral for a tame harmonic bundle}
\begin{itemize}
\item[i)] Let $(E,h,\nabla)\to (Y,\omega)$  be a harmonic Higgs bundle of rank $r$ with bounded Higgs field. Let $p:X\to Y$ be a Galois covering such that $p^{*}(\omega)$ is of bounded geometry with positive injectivity radius.
Then
\begin{align}\chi_{(2)}(X,(E,h)\otimes \Omega^{p}_{X},\dbar)=r\int_{Y}\Todd (T(Y))\Chern(\Omega^{p}_{Y})= r\chi_{(2)}(X,\Omega^{p}_{X},\dbar)\,.
\end{align}
\item[ii)] If moreover $Y=\bar Y\setminus D$, and  $\omega$ is of Poincaré type along the normal crossings divisor $D$, and if $(E,h,\nabla)\to (Y,\omega)$ is a tame nilpotent harmonic bundle by assumption, then
\begin{align}\chi_{(2)}(X, (E,h)\otimes \Omega^{p}_{X},\dbar)=r\int_{\bar Y}\Todd (T\bar Y(-\log(D)))\Chern(\Omega^{p}_{\bar Y}(\log(D)))\,.\end{align}
\end{itemize}
\end{theorem}
\begin{proof}
The hypotheses made about the harmonic bundle implies that the pullback bundle by $p$ has bounded geometry (see Sec.~\ref{higgs bundle of bounded geometry}). The Chern character is multiplicative, hence $\Ch(E\otimes \Omega^{p})=\Ch(E)\wedge \Ch(\Omega^{p})$. But the higher degree Chern forms of the hermitian holomorphic bundle $(E,h)$ vanish: the Chern curvature is a commutator $\Theta_{h}=-[\theta,\theta^{*}]$ with $\theta$ the Higgs field. Granted the relations $\theta\wedge\theta=0$ and $ \theta^{*}\wedge\theta^{*}=0$, one infers that \begin{align}(-1)^{k}\Trace(\Theta_{h}^{k})=\Trace((\theta\theta^{*})^{k})+\Trace((\theta^{*}\theta)^{k})\,,\end{align} but
\begin{align}\Trace((\theta\theta^{*})^{k})=\Trace([(\theta\theta^{*})^{k-1}\theta,\theta^{*}])=0\end{align}
since the trace vanishes on commutators. Hence $\Ch(E)=r$ and
\begin{align}
\chi_{(2)}(X,p^{*}(E,h)\otimes \Omega^{p}_{X},\dbar)=r\int_{Y}\Todd (T(Y))\Chern(\Omega^{p}_{Y})\,.
\end{align}
The last formula is proved as in Cor.~\ref{characteristic logarithmic}
\end{proof}

\subsubsection{Computations of the Todd logarithmic classes}
The Todd class is additive on exact sequences: If $0\to E_{1}\to E\to E_{3}\to 0$ is an exact sequence of vector bundles, then $\Todd(E)=\Todd(E_{1})\wedge \Todd(E_{2})$.
Let $N_{D_{i}}$ be the normal bundle of $D_{i}$ in $\bar Y$. The logarithmic tangent bundle fits into the following exact sequence of coherent analytic sheaves
\begin{align}
0\to T\bar Y\to T\bar Y(-\log(D))\to \oplus_{i}N_{D_{i}}\otimes \O_{D_{i}}\to 0\,.
\end{align}
Hence $$\Todd(T\bar Y(-\log(D)))= \Todd(T\bar Y)\cdot\Todd( \oplus_{i}N_{D_{i}}\otimes \O_{D_{i}})\,,$$
where the Todd class is extended to elements of the Grothendieck Ring $K_{0}(X)$.

The standard sequence $0\to \O\to [D_{i}]\to N_{D_{i}}\otimes \O_{D_{i}}\to 0$ implies that $\Todd( \oplus_{i}N_{D_{i}}\otimes \O_{D_{i}})=\Pi_{i}\Todd([D_{i}])$, hence
\begin{align}
\Todd(T\bar Y(-\log(D)))= \Todd(T\bar Y)\cdot\frac{1}{(1+\frac{c_{1}(D_{1})}{2})\cdot\ldots\cdot(1+\frac{c_{1}(D_{r})}{2})}\,.
\end{align}

Similar computations can be done for the Chern character. In principle this approach allows the computation of the logarithmic characteristic integrals.

\section{Main example: Tame Nilpotent harmonic Higgs bundles}
\subsection{Definition of harmonic Higgs bundles and associated Laplacians}
Following the exposition of Simpson \cite{Sim92} (see also Sabbah \cite{Sab00}),
we consider a flat smooth complex vector bundle $(V,\nabla)\to X$ corresponding to a representation $\rho:\Pi_{1}(X,x_{0})\to GL_{d}(\CC)$.
In general such a representation is not conjugate to a unitary one, and there does not exist a metric $h$ such that $\nabla$ is a metric connection.
\begin{lemma}Let $h$ be a Hermitian metric on $V$. There exists a unique metric connection $D_{h}=\nabla-\vartheta_{h}$ such that $\vartheta'_{h}=\nabla^{1,0}-D_{h}'$ and $\vartheta''_{h}=\nabla^{0,1}-D_{h}''$ are $h-$adjoint.
\end{lemma}
Hence any flat connection $\nabla=D_{h}+\vartheta_{h}$ is equal to the sum of a metric connection and a real self-adjoint field. Flatness and type considerations imply $[D'_{h},\theta'_{h}]=0, [D''_{h},\theta''_{h}]=0$.
\begin{definition}The metric $h$ on a flat bundle $(V,\nabla)$ is called harmonic\footnote{
The term harmonic is explained as follows (\cite[p.16]{Sim92}): Let $\rho:\pi_{1}(X)\to GL_{n}(\CC)$ be the representation defined by the given flat bundle. A metric on $V$ can be considered as a $\rho-$equivariant map $\tilde h:\tilde X\to GL_{n}(\CC)/U(n)$. The differential of $\tilde h$ is given by $\vartheta$ and $\tilde h$ is a harmonic map, iff $D^{*}_{g}\vartheta=0$. 
}
if the operator $D''_{h}+\vartheta'_{h}$ has zero square, ie
\begin{align}
D''^{2}_{h}=0 &&\quad D_{h}''\vartheta'_{h}+\vartheta'_{h}D_{h}''=0 && \vartheta'_{h}\wedge \vartheta '_{h}=0
\end{align}
\end{definition}
If the metric $h$ is harmonic, then the subsheaf $\Ker(D''_{h})$ of the sheaf of smooth sections of $V$ defines a holomorphic structure $E\to X$ on $V\to X$ such that $\vartheta'_{h}$ is a holomorphic $\End(E)-$valued form with $\vartheta'_{h}\wedge \vartheta'_{h}=0$.

The converse construction holds:
\begin{definition} (Cf.\ Simpson \cite{Sim92})\label{definition of harmonic higgs bundle}
\begin{itemize}
\item[i)] A Higgs bundle over a complex manifold $X$ is a holomorphic vector bundle $(E,\dbar)\to X$ together with a holomorphic map $\theta: E\mapsto E\otimes \Omega^{1}_{X}$ such that $\theta\wedge\theta=0$ in $\End(E)\otimes \Omega^{2}_{X}$. Define $D'':=\dbar+\theta$. Then ${D''}^{2}=0$. 
\item[ii)] (\cite[p.18]{Sim92}) Let $K$ be a Hermitian metric on a Higgs bundle, let $\partial_{K}+\dbar$ be the Chern connection, and $\Theta(K)$ its curvature form. Let $\theta_{K}^{*}$ be the adjoint of $\theta$ with respect to $K$ (see the definition below). Let $D_{K}=(\partial_{K}+\dbar)+\theta+\theta^{*}_{K}$ be the Higgs connection. The metric $K$  is called harmonic, if its Higgs connection is flat: $$F_{K}:=D_{K}^{2}= \Theta(K)+[\partial_{K},\theta]+[\dbar,\theta^{*}_{K}]+[\theta,\theta^{*}_{K}]=0\,.$$  \end{itemize}
\end{definition}
A harmonic Higgs bundle $(E,\dbar, \theta,K)$ defines a locally constant sheaf $\underline{\mathcal{E}}=\Ker(D_{K})\to X$. Let $\underline{E}\to X$ be the associated vector bundle with constant transition functions, and let $\nabla$ be the flat connection on the sheaf of its smooth sections $\C^{\infty}\otimes_{\CC}\underline{\mathcal{E}}$. Then $\underline{E}$ and $E$ are isomorphic as smooth bundles. Let $h$ be the metric induced by $K$. Then  $(\underline{E},\nabla,h)$ is a flat bundle with harmonic metric in the previous sense, and the corresponding holomorphic bundle given by the harmonic metric is $(E,\dbar)$. Hence a harmonic Higgs bundle and a harmonic flat bundle are equivalent constructions, and the Higgs field $\theta$ is equal to the $(1,0)-$part  $\vartheta'_{h}$ of the real self-adjoint field constructed above. 
After this equivalence is granted, the harmonic metrics on $\underline{E}$ and $E$ will be denoted by the same letter.

\begin{example}\label{example}
\begin{itemize}
\item[i)]A harmonic line bundle is given by a unitary line bundle and a closed holomorphic one form. A unitary bundle, i.e. a flat hermitian holomorphic bundle, $(E,\dbar,h)$ defines a harmonic bundle  with trivial Higgs field. Let $\alpha$ be a closed holomorphic one form. Multiplication by $\alpha$ defines a diagonal endomorphism $\alpha\in \End(E)\otimes \Omega^{1}$. One checks that $(E,\dbar,\alpha,h)$ is a harmonic bundle.
\item[ii)] Following Hitchin \cite{Hit87}, let $(C,\omega)$ be a hyperbolic Riemann surface of constant sectional curvature -$4$. Let $L\to X$ be a square root of the canonical bundle $K$, i.e.\ $L\otimes L\simeq K$. 
Let $1\in \Hom(L,L^{-1})\otimes K\sim \CC$ be the section defined by the above isomorphism. Then $L\oplus L^{-1}\to C$ with Higgs field $1$ is a Higgs bundle. The hyperbolic metric induces a metric on $L\oplus L^{-1}$. A direct computation shows that the Higgs bundle $(L\oplus L^{-1}, h,1)$ is harmonic. The $n-$th symmetric power of this harmonic bundle is $\oplus_{i+j=n}L^{\otimes i-j}$ with Higgs field $\oplus1\in \oplus \Hom(L^{i-j},L^{i-j-2})\otimes K$ and diagonal metric.
\end{itemize}
\end{example}
\subsubsection{First order Kähler identities.}
Let $(E,\dbar,\theta,h)\to X$ be a harmonic bundle with flat connection $\nabla$. Let $D_{h}=D^{(1,0)}_{h}+\dbar$ be the Chern connection of $(E,\dbar,h)$. Let $\theta^{*}_{h}$ be defined by
\begin{align}
\forall e,e'\in E_{x}, \, (\theta^{*}_{h}e,e')_{h}=(e,\theta e')_{h} \in T^{*(0,1)}_{x}X
\end{align}
Then $\theta$ and $\theta^{*}_{h}$ are parallel with respect to the Chern connection.
Set
\begin{align}
 D''_{\Higgs}=\dbar+\theta \, ,  &\quad D'_{h,\Higgs}=D^{(1,0)}_{h}+\theta^{*}_{h}\, ,
 &\nabla^{c}_{\Higgs} = i(D''_{\Higgs}-D'_{h,\Higgs}) \,.
\end{align}

Let the operators $\nabla, D''_{\Higgs},$ etc.\ act on $E-$valued smooth forms with compact support. The metric $h$ is harmonic, hence each of the above operators has vanishing square and by definition
 \begin{align}
[D'_{h,\Higgs},\,D''_{\Higgs}] &= D'_{h,\Higgs}D''_{\Higgs}+  D''_{\Higgs}      D'_{h,\Higgs}=F_{h,\Higgs} =0\,.
 \end{align}

  Their formal adjoints with respect to the global scalar product $\int_{X}(.,.)_{E\otimes \Lambda^{.} T^{\CC}X}\omega^{n}$ that is induced by a Kähler metric $\omega$ are given by the usual first order Kähler identities (\cite[p.15]{Sim92}).
\begin{align}
 (\theta)^{\star} & =i[\theta^{*}_{h},\Lambda]   &(\theta^{*}_{h})^{\star}&=-i [\theta,\Lambda]    \\
(D'_{h,\Higgs})^{\star} &=  i [\Lambda,D_{\Higgs}'']      & (D_{\Higgs}'')^{\star}&=-i [\Lambda, D'_{h,\Higgs}] \\
(\nabla)^{\star}&=   [\Lambda,\nabla^{c}_{\Higgs}]   & (\nabla^{c}_{\Higgs})^{\star}&=- [\Lambda, \nabla]
\end{align}

These define Kodaira's Laplacians:
\begin{align}
&\Delta=(\nabla+\nabla^{\star})^{2}    &\Delta^{c}_{\Higgs}=(\nabla^{c}_{\Higgs}+\nabla^{c\star}_{\Higgs})^{2} \\
&  \Delta_{\Higgs}''=(D_{\Higgs}''+(D_{\Higgs}'')^{\star})^{2}    &\Delta_{\Higgs}'=(D_{\Higgs}'+(D_{\Higgs}')^{\star})^{2} \,.
\end{align}

Then
\begin{align}
\Delta=\Delta^{c}_{\Higgs}=2\Delta_{\Higgs}''=2\Delta_{\Higgs}'\,.
\end{align}
Assume that the Kähler manifold $(X,\omega)$ is complete, then the formal adjoint and the Hilbertian adjoint are equal. Moreover (see \cite[Cor.~6]{AndVes})
\begin{align}
\Dom(\Delta)\subset \Dom(D+D^{*}),& \quad  \Dom(\Delta'_{\Higgs})\subset \Dom(D'_{\Higgs}+(D'_{\Higgs})^{*}),\\ \Dom(\Delta''_{\Higgs})\subset& \Dom(D''_{\Higgs}+(D''_{\Higgs})^{*})\,.
\end{align}

\subsection{Condition for the existence of a harmonic metric}

Let $\bar Y$ be a complex manifold and $D$ be a normal crossings divisor. Let $(E,\dbar,\theta,h)\to Y:=\bar Y\setminus D$ be a harmonic bundle. Let $P\in D$ and $D(0,1)^{n}\to U$ be a holomorphic coordinate chart centered at $P$ such that $D\cap U=\{z_{1}\cdot\ldots\cdot z_{k}=0\}$. In these coordinates let $\theta=\sum_{1\leq j\leq k }f_{j}\frac{dz_{j}}{z_{j}}+\sum_{k+1\leq j\leq n}g_{j}dz_{j}$.
\begin{definition}\label{tame}
The harmonic bundle $(E,\dbar,\theta,h)\to Y:=\bar Y\setminus D$ is said to be tame, if the coefficients of the characteristic polynomials $\det(t-f_{j})$ and $\det(t-g_{j})$ are holomorphic on $U$.
\end{definition}
We note that such a harmonic bundle is tame, iff there exists a holomorphic bundle $\tilde E\to \bar Y$ and a regular Higgs field $\tilde \theta\in \tilde E\otimes \Omega^{1}_{\bar Y}(\log D)$ such that $(\tilde E,\tilde \theta)_{|Y}=(E,\theta)$ (see e.g. \cite[Lemma 22.1]{Moc07}). In fact the sheaf $\mathcal{E}_{b}\to \bar Y$ ($b\in \RR^{k}$ fixed) defined by the prolongation by increasing order is coherent and locally free. The logarithmic estimate for the Higgs field implies that $\theta$ induces a sheaf homomorphism $\theta:\mathcal{E}_{b}\to \mathcal{E}_{b}\otimes \Omega_{1}(\log(D))$  (see \cite[Thm.~2, p.738]{Sim90}).

The main theorem of Simpson, Mochizuki, and Jost-Zuo is the following (see \cite[Thm.~1.19, p.8]{Moc07} for the relevant definitions).
\begin{theorem}Let $\bar Y$ be a complex projective manifold.
Let $(E,\nabla)$ be a semisimple flat bundle over $ \bar Y\setminus D$, i.e. the monodromy representation $\rho:\pi_{1}(\bar Y\setminus D)\to \Gl(\CC^{n})$ is semisimple. Then there exists a  pure imaginary tame harmonic metric $h$ on $E$, which is unique up to a flat endomorphism of $E$.
\end{theorem}
%
\subsection{Higgs bundle of bounded geometry}\label{higgs bundle of bounded geometry}
The combination of the following theorems proves that a tame, nilpotent, harmonic bundle on a Zariski open set of a Kähler manifold $Y=\bar Y\setminus D$ lifts as a bundle of bounded geometry to a Poincaré covering of $(Y,D)$.
\begin{theorem}(Simpson, Mochizuki) Let $(E,h)\to Y$ be a tame, harmonic Higgs bundle on a quasi projective manifold $Y=\bar Y\setminus D$ where $D$ is a normal crossings divisor. Assume the residue of the Higgs field is nilpotent (i.e.\ the associated representation is unipotent at infinity). Then the Higgs field is bounded with respect to the Poincaré metric.
\end{theorem}
We quote the reference \cite[Cor.\ 22.6]{Moc07}: Let $g_{0}$ be the Poincaré metric on the punctured disc. Assume that the Higgs field is given by $\theta:=f_{0}\frac{dz}{z}$ with $f_{0}$ holomorphic. Let $$
t(\theta)=\sum_{\alpha\in Spec(f_{0}(0))}m(\alpha)|\alpha|^{2}
$$
be the sum  according to multiplicities of the square of the eigenvalues of $f_{0}(0)\in \End(E)_{z=0}$. Then
\begin{align}
\left| \, |\theta|_{h,g_{0}}-2t(\theta)(-\log|z|^{2})^{2}\, \right|\leq C_{0}\,.
\end{align}
Hence the Higgs field has a logarithmic divergence, iff  $f_{0}(0)$ is not nilpotent.
\begin{theorem} Let $\CC^{n}$ be equipped with the standard Euclidean metric $g_{e}$.
Let $C$ be a positive constant. The set $A$ of harmonic Higgs bundles $(E,\dbar_{E},\theta_{E},h_{E})\to B(0,1)$ such that \begin{align}||\theta||_{\infty,h_{E},g_{e}}<C\label{bounds condition on the Higgs field}\end{align} is compact in the smooth topology on $B(0,1)$. Hence, let $F_{h_{E}}$ be the curvature of the harmonic Higgs bundle $(E,\dbar_{E},\theta_{E},h_{E})$ which belongs to $A$, then
\begin{align}
\forall k\in \NN, \exists C_{k}>0 \text{ such that }\quad ||\nabla^{k}F_{h_{E}}||_{\infty, \bar B(0,\frac{1}{2}), h_{E},g_{e}}\leq C_{k}\,.
\end{align}
\end{theorem}
\begin{proof} Let $v$ be a flat frame for $E\to B(0,1)$. Let $H(x)=(h_{E}(v_{i},v_{j}))_{i,j}(x)$, and assume $v$ is chosen such that $H(0)=\Id$. Let $PH(r)$ be the set of positive definite matrices of rank $r$. For $H\in PH(r)$, let $(.,.)_{H}$ be the invariant hermitian product on $T_{H}PH(r)$.
Let $d$ be the invariant distance in the set of positive definite matrices. From  \cite[Sec. 21.2, and $(21.19)$]{Moc07}
\begin{align}
d(\Id,H(x)) &\leq \int_{0}^{1}||\partial_{t}(H(tx))||_{H(tx)}dt \\
                   &=\int_{0}^{1}\Tr(H^{-1}(tx)dH(tx)(x) H^{-1}(tx)dH(tx)(x))^{\frac{1}{2}}dt \\
                  & \leq 8\int_{0}^{1}||\theta||_{h,g_{e}} ||x|| dt
                   \leq 8 ||x|| C\leq 8C\,.
\end{align}
Recall that $d(\Id,H)=\Big(\sum_{j}\log(\lambda_{j})^{2}\Big)^{\frac{1}{2}}$, where $(\lambda_{j})_{1\leq j\leq r}$ is the set of eigenvalues of $H$ (see \cite[Chap.\ 4, Sec.\ 1]{Kob87}). The above estimates imply that $\forall x\in B(0,1)$, $||H(x)||\leq C'$ and $||H^{-1}(x)||\leq C$. 
We claim that  condition 2.87 and 2.91 of  \cite[Sec.\ 2.11]{Moc07} are satisfied for the set of harmonic Higgs bundles with (\ref{bounds condition on the Higgs field}) above: By hypothesis the frame $v$ is flat, hence the flat connection matrix $A$ is vanishing.  The Higgs field is bounded, hence also its adjoint. We showed that $H$ and $H^{-1}$ are universally bounded. We conclude the proof using \cite[Prop.\ 2.96]{Moc07}.
 \end{proof}

\begin{corollary} On a manifold of bounded geometry, a harmonic Higgs bundle, whose Higgs's field is bounded  is of bounded geometry.
\end{corollary}

\subsection{The \texorpdfstring{$\partial\dbar-$}{ddbar-}lemma for harmonic Higgs bundles}\label{ddbar lemma for higgs bundles}
Let $(X,\omega)$ be a complete Kähler manifold of finite volume and bounded curvature such that the injectivity radius is  strictly positive. Let $p:(X,\omega)\to (Y,\omega_{0})$  be a Galois covering with $\Gal(p)=\Gamma$, and let $\omega_0$ be a Kähler form on $Y$ such that  $\omega=p^{*}(\omega_{0})$. Then the Galois $\partial\dbar-$lemma is true for harmonic Higgs bundles of bounded curvature. We adapt the presentation given by Simpson \cite{Sim92} for Higgs bundles on compact manifolds.
\begin{lemma}[Principle of two types]
Let $(E,\dbar_{E},h,\theta)\to X$ be a $\Gamma-$equivariant harmonic bundle of bounded curvature. Then its Kodaira Laplacian is $\Gamma-$Fredholm. Moreover, let $\alpha$ be a square integrable form such that
\begin{align}D'_{\Higgs}\alpha=D''_{\Higgs}\alpha=0\,.\end{align}
 Assume that $\alpha$ is orthogonal to the space of $\Delta-$harmonic forms. 
Then there exists $r\in \Vn(\Gamma)$ almost invertible such that $r.\alpha$ is $D'_{\Higgs}D''_{\Higgs}-$exact.
\end{lemma}
\begin{proof}
It is similar to the proof of Lemma~\ref{Galois ddbar lemma}. The hypothesis of pure type is replaced by the hypothesis that the form is $D'_{\Higgs}-$closed and $D''_{\Higgs}-$closed. The various Kodaira Laplacians are $\Gamma-$Fredholm self-adjoint elliptic operators. From Corollary~\ref{elliptic lifting} we conclude that if $\alpha\in (\Ker(\Delta))^{\perp}$, then there exists $r\in \Vn(\Gamma)$, almost invertible, and $u\in \Dom(\Delta)\subset \Dom(D+D^{*})$ such that
\begin{align*}
r.\alpha &=\Delta u\in \Ker(D'_{\Higgs})\cap \Ker(D''_{\Higgs})\\
              &=\Delta''_{\Higgs} u=D_{\Higgs}''(D_{\Higgs}'')^{*}u   + (D_{\Higgs}'')^{*}D_{\Higgs}''u\Rightarrow (D_{\Higgs}'')^{*}D_{\Higgs}''u=0\Rightarrow D_{\Higgs}''u=0\\
              &=\Delta'_{\Higgs} u=D_{\Higgs}'(D_{\Higgs}')^{*}u   + (D_{\Higgs}')^{*}D_{\Higgs}'u\Rightarrow (D_{\Higgs}')^{*}D_{\Higgs}'u=0\Rightarrow D_{\Higgs}'u=0\,.
  \end{align*}
 The first order Kähler identities imply
 \begin{align*}
              r.\alpha&= D_{\Higgs}'  i [\Lambda,D_{\Higgs}''] u= -D_{\Higgs}'  i D_{\Higgs}''  \Lambda u\,.
\end{align*}
\end{proof}
\subsubsection{}\label{tensor product with affiliated operators}
Let $\U(\Gamma)$ be the ring of operators affiliated with $\Vn(\Gamma)$ (see \cite{MurVon,Luc}). It is a ring of quotients of $\Vn(\Gamma)$ with respect to the multiplicative set of elements which are injective with dense range.
As in \cite{Sim92}, we may state that the DeRham complex is formal and quasi-isomorphic to the Dolbeault complex, after tensoring with $\U(\Gamma)$.
In the following lemmas, the group $\Gamma$ acts by unitary operators on various Hilbert space $X$. This action defines a structure of $\VN(\Gamma)-$module on $X$ and the tensor product $X\otimes \U(\Gamma)$ means $X\otimes_{\VN(\Gamma)}\U(\Gamma)$. This tensor product by a ring of quotient allows to hide the twist by $r\in \VN(\Gamma)$ in the various Galois $\partial\dbar-$lemmas: $x\otimes 1_{\U(\Gamma)}=r.x\otimes r^{-1}$ in $X\otimes_{\VN(\Gamma)}\U(\Gamma)$.
\begin{lemma}
Let $A^{p,q}_{(2)}(E)$ be the set of square integrable sections of $\Lambda^{p}T^{*(1,0)}\otimes \Lambda^{q}T^{*(0,1)}\otimes E$ on $(X,\omega)$. Let $D'_{\Higgs},D''_{\Higgs}$ etc.\ act as closed densely defined unbounded operators on $A^{\tbullet}_{(2)}(E)=\oplus_{p+q=.} A^{p,q}_{(2)}(E)$.
Let $\H^{l}_{\nabla(2)}(E)$be the space of square integrable $\Delta-$harmonic $E-$valued forms. Let $H^{\tbullet}_{DR(2)}(X,E)$ be the cohomology of the complex $ (A^{\tbullet}_{(2)}(E),\nabla)$ and $H^{\tbullet}_{Dol(2)}(X,E)$ be the cohomology of the complex $ (A^{\tbullet}_{(2)}(E),D''_{\Higgs})$.

Then
\begin{align*}
(\Ker(D'_{\Higgs})\otimes \U(\Gamma),D''_{\Higgs}\otimes 1_{\U(\Gamma)}) &\to (A^{\tbullet}\otimes \U(\Gamma),D\otimes 1_{\U(\Gamma)})\\
(\Ker(D'_{\Higgs})\otimes \U(\Gamma),D''_{\Higgs}\otimes 1_{\U(\Gamma)}) & \to   (A^{\tbullet}\otimes \U(\Gamma),D''_{\Higgs}\otimes 1_{\U(\Gamma)})\\
(\Ker(D'_{\Higgs})\otimes \U(\Gamma),D''_{\Higgs}\otimes 1_{\U(\Gamma)}) &\to   ( H^{\tbullet}_{DR(2)}(X,E)\otimes \U(G),0 )\\
(\Ker(D'_{\Higgs})\otimes \U(\Gamma),D''_{\Higgs}\otimes 1_{\U(\Gamma)}) &\to  (H^{\tbullet}_{Dol(2)}(X,E)\otimes \U(G),0)\\
(\Ker(D'_{\Higgs})\otimes \U(\Gamma),D''_{\Higgs}\otimes 1_{\U(\Gamma)}) & \to (\H^{\tbullet}_{\nabla(2)}(E)\otimes \U(G),0)\\
\end{align*}
are quasi-isomorphisms
\end{lemma}
\begin{proof} The statement can be proven as in \cite{Sim92}, once the above principle of two types is known.
\end{proof}

We define the $l^{2}-$hypercohomology of  Higgs bundles in order to rephrase the above lemma as in Green-Lazarsfeld's paper \cite{GreLaz}.
\begin{definition}
Let  $(E,\dbar,\theta,h)\to (X,\omega)$ be a Higgs bundle on a Kähler manifold.  Assume its Higgs field is bounded so that its Chern connection is bounded. Consider the double complex of $\Gamma-$modules of Sobolev spaces (where $k$ is big enough)
\begin{align}
(W^{k-p-q}(X,\Lambda^{p,q}T^{*}X\otimes E),\theta,\dbar)\,.
\end{align}
The $l^{2}-$hypercohomology of the Higgs bundle, noted $\HH^{\tbullet}_{(2)}(X,(E,\theta),\omega)$, is by definition the cohomology of the total complex $(W^{k-\tbullet}(X,\Lambda^{\tbullet}T^{*}X\otimes E,D''_{\Higgs}))$.
\end{definition}

\begin{theorem} With the hypotheses of \ref{ddbar lemma for higgs bundles},
let  $(E,\dbar,\theta,h)\to (Y,\omega_{0})$ be a harmonic Higgs bundle with bounded Higgs field.
 The first spectral sequence of the Double complex
(Hodge to De Rham spectral sequence)
\begin{align}(E_{1}^{p,q},d_{1})=(H^{p,q}_{\dbar(2)}(X,E),\theta)\otimes \U(\Gamma)\abupt \HH_{(2)}^{p+q}(X,(E,\theta),\omega)\otimes \U(\Gamma)\,.
\end{align}
degenerates at $E_{2}$ and abuts to $H^{\tbullet}_{DR(2)}(X,E)\otimes \U(\Gamma)$. Let $\pi$ be the projection onto the harmonic space, then $\Gr^{p}_{F}\HH^{p+q}_{(2)}(X,(E,\theta),\omega))\otimes \U(\Gamma)$ is isomorphic to $$ \frac{\Ker(\pi\theta: \H^{p,q}_{\dbar(2)}(X,E)\to  \H^{p+1,q}_{\dbar(2)}(X,E) )}{\Ran(\pi\theta: \H^{p-1,q}_{\dbar(2)}(X,E)\to  \H^{p,q}_{\dbar(2)}(X,E))}\otimes \U(\Gamma)\,.$$
%
Hence \begin{align}\oplus_{i+j=l}H^{i}(\ldots\overset{\theta}{\to} H^{(\tbullet,j)}_{\dbar(2)}(X,E)\overset{\theta}{\to} \ldots)\otimes \U(\Gamma) \to H^{l}_{DR(2)}(X,E)\otimes \U(\Gamma)\,
\end{align}

is an isomorphism.
\end{theorem}

\begin{corollary}\label{Green-Lazarsfeld for unitary bundles} Let $E\to Y$ be a flat unitary bundle and let $\theta\in H^{0}(Y,\Omega^{1}\otimes \End(E))$ be a bounded flat Higgs field (i.e.\ $\nabla \theta=0$ and $\theta\wedge\theta=0$). Then the spectral sequence
\begin{align}
(H^{q}_{\dbar (2)}(X,E\otimes \Omega^{p}),\theta)\otimes \U(\Gamma)\Rightarrow \HH^{p+q}_{(2)}(X,(E,\theta),\omega)\otimes \U(\Gamma)
\end{align}
degenerates at $E_{2}$.
\end{corollary}
\begin{proof}
Either one mimics the proof given by Green and Lazarsfeld using that the Galois $\partial\dbar-$lemma is true for unitary bundles, or one notices that the hypotheses imply that $(E,h,\theta)$ is a harmonic bundle.
\end{proof}
Flatness of the Higgs field is satisfied, if the data are pullbacks of a logarithmic situation. More precisely, if $E\to Y$ is a flat unitary bundle over $Y=\bar Y\setminus D$ with $D$ a normal crossings divisor, then let $\bar E\to \bar Y$ be the canonical extension of $E$.  Assume that $\theta$ is the restriction of a form in $H^{0}(\bar Y,\dlog{1}{D}\otimes \bar E)$ to $Y$. Then $\theta$ is flat (see e.g. \cite[Thm.\ 1.3]{Ara97}). If $E$ is trivial, this is a classical consequence of Deligne's mixed Hodge theory. Note, however, that the boundedness assumption implies that the residues of $\alpha$ are nilpotent.
\medskip

The study of the above cohomology groups will be completed in the last section in the case of a compact or Stein manifold $Y$.

\section{homotopy invariance and convergence}\label{invariance of the Betti numbers}
The following theorem is one of the main points in the Cheeger-Gromov theory. The pullback of a Riemannian metric, a Hermitian bundle or a local system by a covering map are denoted by the same letter.
\begin{theorem} Let $(X,g)\to (Y,g)$ be a $\Gamma-$Galois covering of a complete Riemannian manifold of finite volume by a manifold of bounded geometry. Let $(E,h,\nabla)\to Y$ be a Hermitian (Riemannian) vector bundle with a flat connection. Assume that it satisfies property $AC^{0}$ (\ref{property A}). Let $\underline{E}=\Ker(\nabla)$ be the local system it defines.
Let a CW-complex structure on $Y$ be given, and assume $K\to Y$ is a CW-sub-complex such that $K\to Y$ is a homotopy equivalence. Let $\tilde K\to K$ be the induced covering of $CW-$complexes. Then
\begin{equation}
\dim_{\Gamma}\H^{i}_{\nabla(2)}(X,E) =b^{i}_{(2)}(\Gamma,K,\underline{E}):=\dim_{\Gamma} \underline{H}^{i}_{(2)}(\tilde{K},\underline{E}),
\end{equation}
and
\begin{equation}
\H^{i}_{\nabla(2)}(X,E) \to \underline{H}^{i}_{(2)}(\tilde{K},\underline{E})
\end{equation}
is a weak isomorphism (injective with dense range).
\end{theorem}

The proof uses three basic ingredients:
\begin{itemize}
\item[1)] De Rham isomorphism theorem for $l^{2}-$cohomology of coverings of compact manifolds with boundary -- it relates the analytic cohomology with a boundary condition to the relative cohomology of a simplicial structure.
\item[2)] Convergence theorem for $l^{2}-$Betti numbers,
\item[3)] Mayer-Vietoris sequence for $l^{2}-$Betti numbers.
\end{itemize}

\subsection{Mayer Vietoris for $l^{2}-$cohomology}
The Mayer-Vietoris sequence (Prop.\ \ref{Mayer Vietoris}) relates the $l^{2}-$cohomology of three manifolds. In the first two sections we dicuss the cohomology of theses manifolds and establish a long exact sequence, in the last paragraph we use the notion of Hilbertian complexes defined in \cite{BruLes}.
Let $\bar U$ be a Riemannian manifold with boundary $\partial U$, and interior $U=\bar U\setminus \partial U$. Let $(E,h)\to (\bar U,g)$ be a hermitian vector bundle. If $s\in \RR$, we denote by $W^{s}(\bar U,E)$ the Sobolev space of order $s$ and by $W_{0}^{s}(\bar U,E)$ the closure of $\C^{\infty}_{0}(\bar U\setminus \partial U,E)$ in $W^{s}(\bar U,E)$ (see \cite[p.284 and p.290]{Tay}).

\subsubsection{Complete Manifolds of bounded geometry with cocompact boundary and cofinite volume}
Let $(\bar M,g)$ be a complete manifold with compact boundary $\partial M$ and of finite volume $\Vol(\bar M)<+\infty$. Let $(\bar U,g)\to (\bar M,g)$ be a Galois covering with Galois group $\Gamma$. We assume that for all $k\in \NN$ the norms $||\nabla^{k}R||_{\infty,\bar U}$ are finite, and if $i_{x}$ is the injectivity radius at $x\in U:=\bar U\setminus \partial U$ then $\exists\, \epsilon,\, c>0$ such that $d(x,\partial U)\geq \epsilon\Rightarrow i_{x}>c$. Then $\bar U$ is a complete manifold with boundary of bounded geometry in the sense of Schick (see \cite{SchNachrichten} or Section~\ref{manifold of bounded geometry}).
The proof of the following lemma is similar to the cocompact case. We give details, since it contains arguments that will often be used.
\begin{lemma}(see \cite{SchPacific,Shu,Ati}) \label{compactness of sobolev embedding}
Let  $(E,h,\nabla_{h})\to (\bar U,g)$ be a $\Gamma-$equivariant bundle of bounded geometry over a manifold of co-finite volume and cocompact boundary (which may be empty). If $s\geq \frac{n}{2}$,  then\begin{align}I:W^{s}(\bar U,E)\to L^{2}(\bar U)\end{align} is $\Gamma-$Hilbert Schmidt.
\end{lemma}
\begin{proof}
Following the notations of Section \ref{manifold of bounded geometry}, let $(\theta_{i})_{i\in I}$ be a uniform partition of unity for $\bar U$ 
Then $supp(\theta_{i})$ is contained in a centered coordinate chart $t_{i}:U(x_{i},\delta)\to \RR^{n}$ or $t_{i}:U(x_{i},\delta)\to \RR^{n}_{x_{n}\geq 0}:=\RR^{n-1}\times [0,+\infty[$ if $x_{i}\in \partial U$ and trivialisation $t_{i}: E_{|U(x_{i},r)}\to U(x_{i},r)\times \CC^{N}$ which are uniformly bounded in the smooth topology. 
We will not indicate the charts and trivializations in the following.

Let $f\in W^{s}(\bar U,E)$. The uniform boundedness of $(\theta_{i})_{i\in \ZZ}$ in the smooth topology implies that $\theta_{i}f\in W^{s}(\RR^{n})$ or $W^{s}(\RR^{n}_{x_{n}\geq 0})$ have a norm bounded by $B_{s}||f||_{W^{s}(B(x_{i},\delta))}$ with $B_{s}$ independent of $i$ and $f$. This is because the Sobolev norms of $W^{s}$ for $s\in \NN$  are defined by an integral of differential operators applied to $f$, hence these are localisable.

Let $C_{s}$ be the norm of the usual injection of $W^{s}(\RR^{n})$ (resp.\ $W^{s}(\RR^{n}_{x_{n}\geq 0})$) to $C^{k}(\RR^{n})$ (resp. $C^{k}(\RR^{n}_{x_{n}\geq 0})$), if $s\geq \frac{n}{2}+k$. Then $||\theta_{i}f||_{C^{k}(\bar U)}\leq B_{s}C_{s}D(A_{0},\ldots,A_{k}) ||f||_{W^{s}(B(x_{i},\delta))}$. This proves that $f\in \U C^{k}(\bar U,E)$ (see the notations in  (\ref{sobolev embedding})), and that moreover it vanishes at infinity.
\medskip

One computes $\Tr_{\Gamma}(I^{*}I)$ (see \cite{Shu,SchPacific,Dix}): let $F$ be a fundamental domain for the covering $\bar U\to \bar M$. It induces an isometry $L^{2}(\bar U,E)\simeq l^{2}(\Gamma)\hat\otimes L^{2}(F,E)$. Consider the polar decomposition $I=U|I|$ of the morphism $I: W^{s}(\bar U,E)\to l^{2}(\Gamma)\hat\otimes L^{2}(F,E)$. Then $U$ is an isometry.
Let $(f_{i})_{i\in \NN}$ be an orthonormal basis of $L^{2}(F,E)$. Then $(g_{i})_{i\in \NN}=(U^{*}f_{i})_{i\in \NN}$ is an orthonormal basis of $W^{s}(\bar U,E)$.
\begin{align}
Tr_{\Gamma}(I^{*}I):=\Tr_{\Gamma}(U(I^{*}I)U^{*})= \sum_{i\in \NN}(U(I^{*}I)U^{*} u_{i},u_{i})_{L^{2}(F,E)}\\
 =\sum_{i\in \NN}(IU^{*}u_{i},IU^{*}u_{i})_{L^{2}(F,E)}=\sum_{i\in \NN}\int_{F}||g_{i}(x)||^{2}dV =\int_{F}\sum_{i\in \NN}||g_{i}(x)||^{2}dV \nonumber
\end{align}
is the integral over a fundamental domain of the Bergman Kernel (square of the norm of the pointwise evaluation map) of $W^{s}(\bar U,E)$. The norm of $m_{s,0}$ (see notation in  (\ref{sobolev embedding})), is bounded by $||m_{s,0}||\leq B_{s}C_{s}D(A_{0},\ldots,A_{k})$ therefore:
\begin{align}
Tr_{\Gamma}(I^{*}I):=\int_{F}\sum_{i\in \NN}||g_{i}(x)||^{2}dV\leq ||m_{s,0}||. \Vol(\bar M)<+\infty \,.
\end{align}

\end{proof}
The notion of Hilbertian complexes is defined in \cite{BruLes} and will be used in the following.

\begin{proposition}\label{Hilbertian versus Fredholm} Let $(E,\nabla,h)\to \bar M$ be a Hermitian (or Riemannian) bundle with a flat connection. Let $(E^{\tbullet},\nabla^{\tbullet},h^{\tbullet})\to \bar M$ be the induced elliptic complex with $E^{i}=\Lambda^{i}T^{*}(\bar M)\otimes E$.

Let $\bar U\to \bar M$ be a Galois covering with Galois group $\Gamma$. Denote the lift to $\bar U$ of the hermitian bundle and of the flat connection by the same letters. Let $H(\bar U,\nabla_{\max})$ be the Hilbertian complex such that
\begin{align}H^{i}:=L^{2}(\bar U,E^{i}), \quad& d^{i}:=\nabla^{i}_{\max}, \quad D^{i}:=\Dom(d^{i})
=\{\alpha\in H^{i} \text{ s.t.\ }& \nabla^{i}\alpha\in  L^{2}(\bar U,E^{i+1})\}\,.
\end{align}
Let $H(\bar U,\nabla_{\min})$ be the Hilbertian complex such that
\begin{align}H^{i}:=L^{2}(\bar U,E^{i}), & \quad d^{i}:=\nabla^{i}_{\min}, \quad D^{i}:=\Dom(d^{i})
=\Adh{\C^{\infty}_{0}(\bar U\setminus \partial U,E)},   
\end{align}
where the closure is taken with respect to ${||.||_{L^{2}}+||\nabla^{i} .||_{L^{2}}}$.

Then $H(\bar U,\nabla_{\max})$ and $H(\bar U,\nabla_{\min})$ are $\Gamma-$Fredholm.
\end{proposition}
\begin{proof} Note that by definition $\nabla^{i}_{max}=((\nabla^{i+1})^{t})_{min}^{*}$ and $(\nabla_{min})^{**}=\nabla_{\min}$. 
We treat the first complex.
Consider the Laplace operator $\Delta=\oplus_{i}( d^{i-1}d^{i-1*}+d^{i*}d^{i})$ associated to the complex  with domain $\oplus_{i}\Dom(d^{i-1}d^{i-1*})\cap \Dom(d^{i*}d^{i})$. Let $\rho$ be the pull-back by the covering map of a defining function for the boundary $\partial M$. We claim that $\Dom(\Delta)$ is equal to \begin{align}
\{u\in W^{2}(\bar U,\oplus_{i}E^{i})\,: \sigma_{\nabla^{t}}(x,d\rho)u_{|\partial U}=0 \,\text{ and } \, \sigma_{\nabla^{t}}(x,d\rho)\nabla u_{|\partial U}=0  \,\}
\nonumber\\
=\{u\in W^{2}(\bar U,\oplus_{i}E^{i})\,: (d\rho \star u)_{|\partial U}=0 \,\text{ and } \, (d\rho \star \nabla u)_{|\partial U}=0  \,\}\,.\label{Hilbertian Domain}
\end{align}

%
%
Call $A$ the above set. Then $A\subset \Dom(\Delta)$, for the vanishing of the boundary values implies that one may integrate by parts.

In order to prove the converse inclusion we use that the unbounded operator $\Delta_{2}$ on $L^{2}(\bar U,\oplus_{i}E^{i})$ defined by $\oplus_{i}\nabla^{i-1}(\nabla^{i-1})^{t}+(\nabla^{i})^{t}\nabla^{i}$ with domain $A$ (see (\ref{Hilbertian Domain})) is a self-adjoint operator according to of \cite[Thm.\ 4.25 ]{SchDissertation}. The inclusion $\Dom(\Delta)\subset \Dom(\Delta_{2}^{*})$ can be checked in an elementary way. Hence self-adjointness implies $ \Dom(\Delta)=A$.

We shall prove that $\Delta$ is $\Gamma-$Fredholm: If $\alpha\in \Ran(1_{[0,\epsilon]}(\Delta))$, then $\alpha\in \cap_{n}\Dom(\Delta^{n})$. But standard elliptic regularity for elliptic boundary value problems (\cite[Thm.~4.15]{SchDissertation}, or \cite[Chap.~2, Thm.\ 5.1, and Rem.\ 7.2, p.202]{LioMag}, \cite[Chap.~5]{Tay}) prove that $\alpha,\ldots,\Delta^{n}\alpha\in \Dom(\Delta)$ implies $\alpha\in W^{2n}(\bar U,\oplus_{i}E^{i})$.

Therefore, $\Ran(1_{[0,\epsilon]}(\Delta))$ maps continuously to $\cap_{n\in \NN}W^{n}(\bar U,\oplus_{i}E^{i})$. One uses the map\linebreak $W^{s}(\bar U,\, E')\to L^{2}(\bar U,\, E')$ is $\Gamma-$trace for any pull-back of a bundle  $E\to \bar M$ (see \cite{SchPacific} in the cocompact case or Lemma \ref{compactness of sobolev embedding} above).
\medskip

Next consider the Hilbert complex $H(\bar U,\nabla_{\min})$: Following the same arguments, one proves that $\Dom((\nabla_{\min})^{*}\nabla_{\min}+\nabla_{\min}(\nabla_{\min})^{*})$ is equal to
\begin{align}\label{minimal extension}
\{u\in W^{2}(\bar U,\oplus_{i}E^{i})\,: (d\rho \wedge u)_{|\partial U}=0 \,\text{ and } \, (d\rho \wedge \nabla^{t} u)_{|\partial U}=0  \,\}\,.\end{align} One concludes the proof as above.
\end{proof}
\begin{remark}The fact that the pullback of an elliptic boundary value problem on a cocompact manifold is $\Gamma-$Fredholm is proved in \cite{SchPacific}. There, the equality of the $l^{2}-$index of the lifted elliptic boundary value problem on $\bar U$ and index of the elliptic boundary problem on $\bar M$ is proved. Here we took some care to prove that the cohomology of the Hilbert complexes are associated to such kind of problems.
\end{remark}
\subsubsection{Complete Manifold of bounded geometry with a co-finite volume}
Although we will only use elliptic complexes of first order differential operators, in this section, we prove that general elliptic complexes define $\Gamma-$Fredholm complexes. 

\begin{lemma}\label{elliptic is fredholm}
Let $P:\C^{\infty}_{0}(X,E)\to \C^{\infty}_{0}(X,F)$ be a $\Gamma-$invariant uniformly elliptic partial differential operator of order $m$ on $X$ between two $\Gamma-$invariant Riemannian bundles of bounded geometry. Then the maximal extension $[P]$ of $P$ is $\Gamma-$Fredholm, and
$P_{s}$ the extension of $P$ from $W^{s}(X,E)$ to $W^{s-m}(X,F)$ is $\Gamma-$Fredholm.
\end{lemma}
\begin{proof}
Let $P:\C^{\infty}_{0}(X,E)\to \C^{\infty}_{0}(X,F)$ (resp.\ $Q:\C^{\infty}_{0}(X,F)\to \C^{\infty}_{0}(X,G)$) be a uniformly elliptic partial differential operators of order $m$ (resp. $m'$) between Riemannian bundles of bounded geometry.
Let $[P]: L^{2}(X,E)\to L^{2}(X,F)$ be the closure of the unbounded operator defined by $P$ with domain $C^{\infty}_{0}(M,E)$. Elliptic regularity and bounded geometry imply (\cite{ShuNantes,Kor}) 
that
$\Dom([P])=W^{m}(M,E)$, the maximal and minimal extension coincide, and $[Q][P]=[QP]$. In particular $\Dom([Q][P])=W^{m+m'}(X,E)$, and the Hilbertian adjoint $[P]^{*}$ of $[P]$ is equal to the closure $[P^{t}]$ of the transposed operator.

First assume that $E=F$ and $P$ is formally positive. Then $[P]$ is positive, and self-adjoint. Since $P$ is a uniformly elliptic operator on a manifold of bounded geometry, the spectral projection $1_{[0,\lambda]}([P])$ maps $L^{2}(X,E)$ continuously to $ \cap_{k\in \NN}\Dom([P]^{k})\subset \cap_{k\in \NN} W^{km}(X,E)$. Bounded geometry implies (\cite{ShuNantes,Roe, SchDissertation,Kor}) that the latter space embeds continuously in $\UC^{\infty}(X,E)$.
Hence the Schwartz Kernel $k_{\lambda}(x,y)$ of $1_{[0,\lambda]}([P])$ is bounded.

Let $a:x\mapsto x^{-1}1_{[1,+\infty[}(x)$. Let $A(P):=a([P])$ be defined through the functional calculus. Then $A(P)$ maps $W^{s}(X,E)$ continuously
to $W^{s+m}(X,E)$, because 
$[P] A(P)(\alpha)=\alpha-1_{[0,1[}(P)(\alpha)$ belongs to $W^{s}(X,E)$, if $\alpha$ does ($1_{[0,1[}([P])$ is a uniform smoothing operator). Let $A(P)_{s}: W^{s}(X,E)\to W^{s+m}(X,E)$ be the induced operator.

Now assume all of the data are $\Gamma-$equivariant, and the quotient manifold $X/\Gamma$ is of finite volume. Let $F$ be a fundamental domain for the $\Gamma-$action. Then \begin{align*}\dim_{\Gamma}1_{[0,\lambda]}([P])(L^{2}(X,E)):=\int_{F}\tr(k_{\lambda}(x,x))dv\end{align*} is finite. Self-adjointness implies that $[P]$ is $\Gamma-$Fredholm and $A(P)$  is an inverse modulo $\Gamma-$trace classes operators.
 Since $B(P):=1_{[0,\lambda]}([P])$ maps $W^{s}(X,E)$ continuously to $\cap_{k\in \NN} W^{km}(M,E)$, it induces a $\Gamma-$trace classe operator $B_{s}(P):W^{s}(X,E)\to W^{s+m}(X,E)$.

Moreover, the induced operators satisfy
\begin{align}
A_{s-m}(P)P_{s}-I=B(P)_{s} \text{, and }P_{s}A_{s-m}(P)-I=B(P)_{s-m},
\end{align}
for these equations hold on the dense subset $\C^{\infty}_{0}(X,E)$.

For a general elliptic operator $P:\C^{\infty}_{0}(X,E)\to \C^{\infty}_{0}(X,F)$, the first point proves that the positive operators $[P^{t}P]=[P]^{*}[P]$ and $[PP^{t}]=[P^{t}]^{*}[P^{t}]$ are $\Gamma-$Fredholm. Hence, by definition, $[P]$ is $\Gamma-$Fredholm. One concludes that $[P]_{s}$ is $\Gamma-$Fredholm using Lemma \ref{from unbounded to bounded}.
\end{proof}

%
Elliptic complexes were defined by Atiyah and Bott \cite{AtiBot}.
\begin{proposition}\label{elliptic complex is fredholm}
Let $(E^\tbullet,h^\tbullet, d^\tbullet)\to (Y,g)$ be an elliptic complex on a complete manifold of bounded volume. Let $X\to Y$ be a Galois covering with Galois group $\Gamma$, and denote the pullback of the elliptic complex by the same letter $(E^\tbullet,h^\tbullet, d^\tbullet)\to (X,g)$ .

Assume $(E^\tbullet,h^\tbullet, d^\tbullet)\to (X,g)$ is of bounded geometry. Then the associated Hilbert complex $(H^\tbullet(E),[d^\tbullet],D^\tbullet)$, and Sobolev complex 
$(\ldots\to W^{s-\sum_{j<i}n_{j}}(X,E^{i})\overset{d^{i}}{\to}  W^{s-\sum_{j<i+1}n_{j}}(X,E^{i+1}) \to\ldots)$ are $\Gamma-$Fredholm.
\end{proposition}
\begin{proof}
We treat first the simpler case where $m_{i}=m$ is constant.

Then $\Delta_{i}:=d_{i}^{t}d_{i}+d_{i-1}d_{i-1}^{t}:\C^{\infty}_{0}(X,E^{i})\to  \C^{\infty}_{0}(X,E^{i})$ is elliptic and positive.
Lemma~\ref{elliptic is fredholm} shows that it is essentially self-adjoint, that for all $\epsilon\geq 0$ the operator  $1_{[0,\epsilon]}([\Delta_{i}])$ is uniformly smoothing and that $a([\Delta_{i}])$ (with $a(x)=x^{-1}1_{[1,+\infty[}(x)$) defines a $\Gamma-$Fredholm inverse of $[\Delta_{i}]$ modulo $\Gamma-$trace operators. Uniform ellipticity implies that $[\Delta_{i}]=[d^{i-1}][d^{i-1}]^{*}+[d^{i}]^{*}[d^{i}]$. Hence the Hilbert complex is $\Gamma-$Fredholm. Lemma~\ref{from unbounded to bounded} implies that the corresponding Sobolev complexes are $\Gamma-$Fredholm.
\medskip

The case of non-constant degrees is proved following Atiyah and Bott (\cite[Sect.~6]{AtiBot}) using Lemma~\ref{elliptic is fredholm}.
\end{proof}

\begin{remark}\label{fundamental remark}
The treatment of general elliptic complexes on manifolds with boundary is not amenable to the above methods. The natural boundary problem given by the associated Hilbertian Laplacian does not produce an elliptic boundary problem unless geometric conditions on the boundary are imposed (see the discussion in \cite[Chap.~10]{Tay}). It is important that this holds for the DeRham complex.
Gromov, Henkin, Shubin \cite{GroHenShu} treated the Dolbeault complex on coverings of strictly pseudoconvex compact complex manifolds with boundary. The case of strictly pseudoconvex manifolds of bounded volume works verbatim.
\end{remark}

\subsubsection{Mayer Vietoris sequence for $l^{2}-$cohomology}\label{notations mayer vietoris}
Variations of the following proposition were stated in \cite{CheGro85-bounds,LotLuc,LucSch} (see also \cite{Car98}  in a related context). Below we give a proof of the Mayer-Vietoris sequence using an exact sequence of Hilbertian complexes (see \cite{BruLes}) as in \cite[Lemma~4.3]{Chee}
Notations: Let $X'$ be a Riemannian manifold with our without boundary.
Let $\underline{H}^{p}_{\nabla (2)}(X',E,h)$ be the reduced cohomology of the maximal extension of $\nabla$ and  $\underline{H}^{p}_{\nabla (2)}(\bar X',\partial X',E,h)$ be the reduced cohomology of the minimal extension of $\nabla$. We will drop the reference to $\nabla,h$ etc.\ if no confusion occurs.
\begin{proposition}\label{Mayer Vietoris}
Let $(E,h,\nabla)\to (Y,g)$ be a Hermitian vector bundle with a flat connection over a complete Riemannian manifold of bounded volume. Let $(X,g)\to (Y,g)$ be a Galois covering with Galois group $\Gamma$, and let $(E,h,\nabla)\to (X,g)$ denote the pullback bundle. Assume that $(E,h,\nabla)\to (X,g)$ is of bounded geometry. Let $\bar M\subset Y$ be a compact manifold with boundary with $\dim\bar M=\dim Y$, and let $\bar U\to \bar M$ be the induced covering, where $\bar U\subset X$.

The following sequence of $\Gamma-$Fredholm complexes (notations of Prop.~\ref{Hilbertian versus Fredholm})
\begin{align}\label{short exact sequence}
0\to H(X\setminus \bar U,\nabla_{\min})\to H(X,\nabla_{\max})\to H(\bar U,\nabla_{\max})\to 0
\end{align}
is exact and induces a long, weakly exact sequence of $\Gamma-$modules of finite $\Gamma-$dimension:
\begin{align}\label{long exact sequence of relative cohomology}
\ldots \to\underline{H}^{i}_{(2)}(X\setminus \bar U, \partial U, E)\to \underline{H}^{i}_{(2)}(X,  E) \to \underline{H}^{i}_{(2)}(\bar U,  E) \to \underline{H}^{i+i}_{(2)}(X\setminus \bar U,\partial U, E)\to \ldots
\end{align}
\end{proposition}

 \begin{proof}
 Note that $\nabla_{\min}=\nabla_{\max}$ on $X$, hence (\ref{short exact sequence}) is well defined. We already proved that these complexes are $\Gamma-$Fredholm.
 We shall prove the exactness of this sequence at the middle term.

 The restriction map $r: X\to \bar U$ induces a surjective map of complexes  $r: H(X,\nabla_{\max})\to H(\bar U,\nabla_{\max})$. Let $e: L^{2}(X\setminus \bar U,\Lambda^{\tbullet}\otimes E)\to  L^{2}(X,\Lambda^{\tbullet}\otimes E)$ be the extension by the zero map (dual of the restriction operator to $X\setminus \bar U$). The kernel of $r$ is isomorphic to the sub-complex $e:K \to H(X,\nabla_{\max})$ with
 \begin{align}
 K^{\tbullet}
          =\{\alpha\in H(X\setminus\bar U,\nabla_{\max}) \text{ s.t. }  \nabla e(\alpha)=e(\nabla \alpha)\}\,.
 \end{align}
 On the last set, the operator acts in the sense of distributions. Hence, an element of $K^{\tbullet}$ has the property that  the extension by zero does not produce a jump of the differential along $\partial U$.  Let $\nabla_{e}:=e^{*}\circ\nabla_{\max}\circ e$ be the induced operator with domain $K^{\tbullet}$.  It is closed and densely defined.

We prove below that $(K^{\tbullet},\nabla_{e})=H(X\setminus \bar U,\nabla_{\min})$. Hence, there exists a weak exact long sequence $(\ref{long exact sequence of relative cohomology})$, because the sequence $(\ref{short exact sequence})$ is exact, and the complexes are $\Gamma-$Fredholm.

The weak exactness of the long exact cohomology sequence under the $\Gamma-$Fredholm hypothesis is a result by Cheeger and Grovov \cite{CheGro85-bounds} (see \cite{Shu} for a detailled proof). Elliptic regularity implies that the harmonic space (\ref{minimal extension}) associated to $H(X\setminus \bar U,\nabla_{\min})$ is \begin{align*}\{\alpha \in W^{1}(X\setminus \bar U,\Lambda^{\tbullet}\otimes E) \, \text{ s.t. }\nabla\alpha=0,\, \nabla^{t}\alpha=0,\, i^{*}(\alpha)=0\}, \end{align*} which is one of the definitions of $\underline{H}^{.}_{(2)}(X\setminus \bar U,\partial U,E)$.

We now prove that $(K^{\bullet},\nabla_{e})=H(X\setminus \bar U,\nabla_{\min})$. By Lemma 2.3 of \cite{BruLes}, it is enough to prove that they have equal Laplace operators. Each of these operators is self-adjoint, hence it is enough to prove that one is an extension of the other. Recall that if $(T_{i},\Dom(T_{i}))$, $(i=1,2)$ are unbounded operators  on the Hilbert space $H$, the operator $T_{1}$ is an extension of $T_{2}$, noted $T_{1}\subset T_{2}$, if $\Dom(T_{1})\subset \Dom(T_{2})$ and $T_{2|\Dom(T_{1})}=T_{1}$.
The Laplace operator $\Delta$ associated to $\nabla_{\min}$ is the operator $\nabla^{t}\nabla +\nabla^{t}\nabla$ with domain given in $(\ref{minimal extension})$. Hence $\nabla_{\min}\subset \nabla_{e}\subset \nabla_{\max}$  is valid on $L^{2}(X\setminus \bar U,\Lambda^{\tbullet}T^{*}\otimes E)$. Moreover on  $L^{2}(X,\Lambda^{\tbullet}T^{*}\otimes E)$, minimal and maximal extension are equal, $\nabla_{\max}^{*}=\nabla^{t}_{\min}$. Therefore, let $u\in \Dom(\Delta_{e})$, $v\in \Dom(\Delta_{\min})$, let $\tilde v\in H^{2}(X,\Lambda^{\tbullet}T^{*}\otimes E)$ be such that $\tilde v_{|\bar U}=v$ (see  \cite[Chap.~4, Lemma~4.1]{Tay}, \cite{SchPacific},\cite[Lemma~1.3]{ShuNantes}), then
\begin{align}
(u,\nabla^{t}\nabla v)=(e(u),\nabla^{t}\nabla \tilde v)=(e(u),(\nabla_{\max})^{*}\nabla\tilde v)=(e(\nabla u),\nabla\tilde v) \\
 =(\nabla u,\nabla v)=(\nabla u,\nabla_{e}v)=(\nabla_{e}^{*}\nabla_{e}u,v)\,.\nonumber
 \end{align}
From $\nabla_{\min}\nabla_{\min}^{*}\subset \nabla_{e}\nabla_{\min}^{*}$, we infer
 \begin{align}
 (u,\nabla\nabla^{t} v)=(u\nabla_{e}\nabla^{t}v)=(\nabla_{e}^{*}u,\nabla^{t} v)=(e(\nabla_{e}^{*}u),\nabla_{\max}^{*}\tilde v)=(\nabla_{e}\nabla_{e}^{*}u,v)\,.
 \end{align}
 The equality $ (u,\Delta_{\min}v)=(\Delta_{e}u,v)$, valid for all $(u,v)\in \Dom(\Delta_{e})\times \Dom(\Delta_{\min})$, and self-adjointness of the operators imply that $\Delta_{e}=\Delta_{\min}$.
\end{proof}

\subsection{DeRham's isomorphism after Dodziuk \texorpdfstring{\cite{Dod}}{} and Schick \texorpdfstring{\cite{SchDissertation}}{}}
 We follow \cite{SchDissertation}. Let $\bar X$ be a complete manifold with boundary  of bounded geometry such that $\partial  \bar X=\bar X\setminus X=X_{1}\sqcup X_{2}$. $X_{1}$ or $X_{2}$ may be empty. Let $i_{k}:X_{k}\to \bar X$ $(k=1,2)$ be the canonical injections. Let $(E,\nabla,h)\to \bar X$ be a hermitian vector bundle of bounded geometry together with a connection (not necessarly Hermitian). Then $i_{k}$ induces a pullback map $i^{*}_{k}:\left(\Lambda^{p}T^{*}\bar X\otimes E\right)_{|\bar X}\to \Lambda^{p}T^{*}X_{k}\otimes E_{|X_{k}}$ such that $i^{*}_{k}(\alpha\otimes e)=i^{*}_{k}(\alpha)\otimes e$ for $\alpha\in (\Lambda^{p}T^{*}X)_{x}$ and $e\in E_{x}$, $x\in X_{k}$. 

Set $ {\mathcal L}^{p}_{r}(X_{1},E):= \{\alpha\in \C^{\infty}_{0}(\bar X,\Lambda^{p}T^{*}\bar X\otimes E);\,\text{ s.t. }i^{*}_{1}(\alpha)=0\}$, where $r$ stands for relative, and
 consider the unbounded operators
 \begin{align*}
 \nabla&:{\mathcal L}^{p}_{r}(X_{1},E)\subset L^{2}(\bar X,\Lambda^{p}T^{*} X\otimes E)\to  L^{2}(\bar X,\Lambda^{p+1}T^{*} X\otimes E)\,.
 \end{align*}
Let $A^{p}(\bar X,X_{1},E)$ be the domain of the closure $\nabla_{r}$ of $\nabla$ with respect to the  norm $|.|^{2}+|\nabla .|^{2}$.
Define
\begin{align*}
\H^{p}(\bar X,X_{1},X_{2},E)&=\{\alpha\in \C^{\infty}(\bar X,\Lambda^{p}T^{*}\bar X\otimes E)\cap L^{2}\text{ s.t. } \\ & \nabla\alpha=0,\,\nabla^{t}\alpha=0,\, i^{*}_{1}(\alpha)=0,\,i^{*}_{2}(\star\alpha)=0\}\,.
\end{align*}

Assume moreover that the connection is flat. Then one checks that
\begin{align}
(L^{2}(\bar X, \Lambda^{.}T^{*}\bar X\otimes E),\nabla_{r}^{.}, \Dom(\nabla_{r}^{.}))=A^{.}(\bar X,X_{1},E))
\end{align}
 is a Hilbert complex (see \cite[Chap.~5]{SchDissertation}). Let
\begin{align*}\underline{H}^{p}(\bar X,X_{1},E):=\Ker(\nabla:A^{p}\to A^{p+1})/\Adh{\Ran(\nabla)}
\end{align*}
be its reduced cohomology.
Then (loc.cit. Thm.~6.2):
\begin{theorem}The inclusion $\H^{p}(\bar X,X_{1},X_{2},E)\to A^{p}(\bar X,X_{1},E)$ induces an isomorphism \begin{align*}\H^{p}(\bar X,X_{1},X_{2},E)\simeq \underline{H}^{p}(\bar X,X_{1},E)\,.\end{align*}
%
\end{theorem}
%
%
%
%
%
With the notations of (\ref{Hilbertian versus Fredholm}), elliptic regularity implies that the harmonic space (\ref{minimal extension}) associated to $H(\bar X,\nabla_{min})$  is equal to $\oplus_{i} \H^{i}(\bar X,\partial X,\emptyset)$, and the harmonic space (\ref{Hilbertian Domain}) associated to $H(\bar X,\nabla_{max})$  is equal to $\oplus_{i} \H^{i}(\bar X,\emptyset,\partial X)$.
\subsubsection{Poincar\'e duality}

By definition, the Hodge star operator exchanges absolute and relative boundary conditions -- it maps harmonic forms to harmonic forms, therefore:
\begin{proposition}[{\cite[Thm.~6.3, p.50]{SchDissertation}}] Let $X$ be a complete Riemannian manifold with boundary $\partial X=X_{1}\sqcup X_{2}$. Let $(E,h)\to \bar X$ be a Hermitian bundle. Then \begin{align*}
\star: \H^{p}(\bar X,X_{1},X_{2},E)\sim \H^{n-p}(\bar X,X_{2},X_{1},E)
\end{align*}
is a well defined isomorphism, which induces an isomorphism $\underline{H}^{p}_{(2)}(\bar X,X_{1},E)\simeq \underline{H}^{n-p}_{(2)}(\bar X,X_{2},E)$.
\end{proposition}
\begin{remark}
The $\sharp$ metric operator (see \cite[Sec.~7.3, p.40]{BerDemIllPet}) maps $p-$forms with values in $E$ to $(n-p)-$forms with values in $E^{*}$ and maps harmonic forms with absolute (resp.\ relative) conditions to harmonic forms with relative (resp.\ absolute) conditions. This gives the Serre Duality for $l^{2}-$cohomology of manifolds with boundary.
\end{remark}

A variation of \cite{Dod,SchDissertation} gives:
\begin{theorem}(DeRham's theorem)\label{DeRham isomorphism}
Let $p: \bar X\to \bar Y$ be a covering of a compact manifold with boundary $\partial Y=Y_{1}\sqcup Y_{2}$. Let $\partial X=X_{1}\sqcup X_{2}$ be the corresponding boundaries of $\bar X$. Let $(J,J_{1},J_{2})$ be a triangulation of $(\bar Y,\partial Y)$, and let $(K,K_{1},K_{2})$ be the lifted triangulation of $(\bar X,\partial X)$.

 Let $(E,h,\nabla)\to \bar X$ be a hermitian bundle with a flat connection of bounded geometry, and let $\underline{E}:=\Ker(\nabla)$ be the local system it defines.

Let $C^{\tbullet}_{(2)}(K,K_{1},\underline{E})$ be the subcomplex of square integrable cochains of $C^{\tbullet}(K,K_{1},\underline{E})$  (the cochains that vanish on $K_{1}$). Then integration over the simplices of $K$ gives an isomorphism

\begin{align*}
I: \H^{p}(\bar X,X_{1},X_{2},E)\to \underline{H}^{p}_{(2)}(\bar X,X_{1},E)\sim \underline{H}^{p}_{(2)}(K,K_{1},\underline{E})\,.
\end{align*}
In particular if $X_{1}=\emptyset$, $I$ induces isomorphisms between:
\begin{itemize}
\item[a)] $\H^{p}_{(2)}(X,\emptyset,X_{2},E)$, the space of harmonic forms with absolute boundary condition,
\item[b)] $\underline{H}_{(2)}^{p}(\bar X,E)$, the reduced $l^{2}-$cohomology of the maximal extension of $\nabla$,
\item[c)] $\underline{H}^{p}_{(2)}(K,\underline{E})$ the reduced $l^{2}-$cohomology of the simplicial complex $K$ with coefficient in $\underline{E}$.
\end{itemize}
If $X_{2}=\emptyset$, $I$ induces isomorphism between:
\begin{itemize}
\item[a)] $\H^{p}_{(2)}(X,X_{1},\emptyset, E)$, the space of harmonic forms with relative boundary condition,
\item[b)] $\underline{H}_{(2)}^{p}(\bar X, X_{1}, E)$, the reduced $l^{2}-$cohomology of the minimal extension of $\nabla$,
\item[c)] $\underline{H}^{p}_{(2)}(K,K_{1}, \underline{E})$ the relative reduced $l^{2}-$cohomology of the simplicial pair $(K,K_{1})$ with coefficients in $\underline{E}$.
\end{itemize}
\end{theorem}

One recall the notion of integration of a $\underline{E}-$valued $k-$form $\alpha$ over simplices. A simplex is a $C^{1}-$map $f$ from the standard simplex $\Delta_{k}$ in $\RR^{k}$ to the manifold.
The pullback of the local system to $\Delta_{k}$ admits a functorial trivialisation  $t:f^{*}(\underline{E})\simeq f^{*}(\underline{E})_{0}\times \Delta_{k}$ where $f^{*}(\underline{E})_{0}$ is the fiber over $0\in\RR^{k}$. By the sheaf isomorphism $C^{k}(X,E)\simeq C^{k}(X)\otimes \underline{E}$,
integration of the vector valued $k-$form $t\circ f^{*}\alpha$ defines the value
\begin{align*}I(\alpha)(f)=\int_{f}\alpha:=f_{0}(\int_{\Delta_{k}}t\circ f^{*}\alpha)\in\underline{ E}_{f(0)}\,.\end{align*}

This already proves the homotopy invariance of $\underline{H}^{\tbullet}(X,X_{1},E)$. We shall give now a further isomorphism based upon sheaf cohomology.
\begin{definition}(see \cite{CamDem,Din2013,Eys})\label{l2 direct image}
\begin{itemize}
\item[i)] Assume that $(E,h,\nabla)=p^{*}(E',h',\nabla')$ is the pullback of a bundle with a flat connection on $\bar Y$. Let $\underline{E'}\to \bar Y$ be the local system it defines, and let $\p(\underline{E'})\to \bar Y$ be its $l^{2}-$direct image sheaf. It is the locally constant sheaf over $\bar Y$ associated to the pre-sheaf $U\mapsto L^{2}(p^{-1}(U),p^{*}(\underline{E'}))$.
\item[ii)] For any coherent analytic sheaf $\F\to Y$, there exists a preasheaf $U\mapsto L^{2}(p^{-1}(U),p^{*}(\F))$ on $Y$. Let $\p\F\to Y$ be the associated sheaf.
\end{itemize}
\end{definition}
 If $X_{1}=\emptyset$ (but not necessary $X_{2}=\emptyset$),
\begin{align}
H^{p}(\bar Y, \p(\underline{E'}))\simeq H^{p}_{(2)}(K,\underline{E})\simeq \underline{H}^{p}_{(2)}(\bar X,E)
\end{align}
is a homotopy invariant, for it is well known that the sheaf cohomology of a local system is a homotopy invariant.
\subsection{Convergence of $l^{2}-$Betti numbers}

\begin{theorem}[Cheeger-Gromov \cite{CheGro85-bounds}] Let $p:X\to Y=X/\Gamma$ be a locally isometric covering of a complete manifold of finite volume by a manifold of bounded curvature and strictly positive injectivity radius. Let $(E,\nabla,h)\to Y$ be a Riemannian bundle of bounded Riemannian curvature equipped with a flat connection. Assume $(E,h)$ satisfies Property~$AC^{0}$ (\ref{property A}). Let $\cup_{j}Y_{j}=Y$ be an exhaustion by smooth compact submanifolds with boundary.

Define
\begin{align}b_{i(2)}(p^{-1}( Y_{k}), p^{-1}(Y_{l}),E)&:=\dim_{\Gamma}\Im(\quad \underline{H}^{i}_{(2)}(p^{-1}(Y_{l}),E)\quad\to \quad\underline{H}^{i}_{(2)}(p^{-1}(Y_{k}),E)\quad)\\
b_{i(2)}(p^{-1}(Y_{k}),E)&:=\dim_{\Gamma}\underline{H}^{i}_{(2)}(p^{-1}(Y_{k}),E)\,.
\end{align}
Then
\begin{align}
\lim_{k} b_{i(2)}(p^{-1}(Y_{k}),E)= \lim_{k} \lim_{l}b_{i(2)}(p^{-1} (Y_{k}),p^{-1}( Y_{l}),E)=&b_{i(2)}(X,E)\,,
\end{align}
moreover
\begin{align}
\lim_{j\to +\infty} \dim_{\Gamma} \underline{H}^{i}_{(2)}(p^{-1}(Y_{j}), \partial p^{-1}(Y_{j}), E)=&\dim_{\Gamma} \underline{H}^{i}_{(2)}(X, E)\,.
%
\end{align}
\end{theorem}
\begin{proof}
We will prove the theorem under the assumption that $X$ and $E$ have bounded geometry, the general case is proved using the approximation theorem by metrics of bounded geometry. Let $X_{j}=p^{-1}(Y_{j})$. Using the Mayer-Vietoris sequence (\ref{Mayer Vietoris}) and simple diagram chasing (see the argument given in \ref{proof of convergence for special exhaustion}), the claim is reduced to
$$
\lim_{j}\dim_{\Gamma}\underline{H}^{i}_{(2)}(X\setminus X_{j},E)=0\,.
$$
We note the following monotonicity:
\begin{lemma}[{\bf Monotonicity lemma}]
Let $Z\overset{i}{\to} X\overset{j}{\to} X'\overset{k}{\to}Z'$ be compact submanifolds of $Y$ with boundary. Let $i^{*}$, $j^{*}$, $k^{*}$ be the induced morphism between the reduced cohomology of the covering spaces $p^{-1}(Z)\to p^{-1}(X)\to p^{-1}(X')\to p^{-1}(Z')$. Then \begin{align}&\Rank_{\Gamma}(i^{*}j^{*}k^{*})\leq \Rank_{\Gamma}(j^{*})\nonumber\\ \text{ i.e. } \quad&b_{i(2)}(p ^{-1}(Z), p^{-1}(Z'),E)\leq b_{i(2)}(p ^{-1}(X), p^{-1}(X'),E)\,.\end{align}
\end{lemma}

Hence, it is enough to prove the theorem for a special exhaustion $(Z_{k})_{k\in\NN}$ of $Y$, by compact submanifolds with compact boundary, for which it is easy to find uniform upper bounds (in $k$) of $||K_{\H^{i}(X\setminus  p^{-1}(Z_{k}),E)}(x)||_{\infty}$, the sup norms of the Bergmann Kernel of the spaces of harmonic $E-$valued $i-$forms. Indeed $\lim_{k}\Vol(Y\setminus Z_{k})=0$ will imply $\lim_{k}b_{i(2)}(X\setminus p^{-1}(Z_{k}),E)=0$.

\subsubsection{Construction of a uniform exhaustion}
Dafermos \cite{Daf} proved that there exist $\eta_{0}>0$ and an exhaustion function $f:Y\to \RR$, with bounded covariant derivative of any order, such that for a sequence $a_{i}\to +\infty$ one has $||\nabla f||_{|f=a_{i}}\geq \eta_{0}$. One may assume that $a_{i+1}-a_{i}\geq 1$. Let $\bar Z_{k}=f^{-1}(]-\infty,a_{k}])$.

For the induced metric on $\partial Z_{k}$, elementary computations show that
\begin{align}
 |\nabla^{i}_{\partial Z_{k}}R_{\partial Z_{k}}|\leq \frac{P_{i}(|\nabla^{1}  f|,\ldots, |\nabla^{i} f|, R_{Y},\nabla R_{Y},\ldots,\nabla^{i}R_{Y})}{|\nabla  f|}\quad{}_{|\partial Z_{k}}\label{induced geometry}
 \end{align}
for some universal polynomial function $P_{i}$ (see e.g.\ \cite{SchDissertation}).  The following lemma applies to the covering manifold $p:X\to Y$ which is assumed to be of positive radius.
\begin{lemma} \label{uniform collar} Let $X$ be a complete manifold of bounded curvature of order two and positive injectivity radius. Then for all $A>0$, there exists positive constants $B,r$ which depend only on the bounds of the geometry, the injectivity radius, and $A$, such that for any smooth oriented hypersurface $W$ in $X$ with second fundamental form bounded by $A$, for any $x\in W$, the map \begin{align}e:B_{T_{x}(W)}(0,r)\times [-B,B]\to X,\quad\quad (y,t)\mapsto \exp^{X}_{\exp^{W}_{x}(y)}(\nu_{\exp^{W}_{x}(y)} t)\end{align} is a diffeomorphism onto its image.

Moreover, if $W=f^{-1}(0)$ for some function such that $||\nabla f||_{|f^{-1}(0)}\geq \eta$ and $||\nabla^{2}f||_{\infty,X}<+\infty$ then, if $0<r'\leq\min(B,\frac{\eta_{0}}{2(||R||_{\infty}+||\nabla^{2}f||_{\infty})})$, the global map $e:W\times [-r',r']\to M$ is an embedding.
\end{lemma}
(Here $\nu_{\exp^{W}_{x}(y)}$ is the inward normal field at $\exp^{W}_{x}(y)\in W$.)
Hence, the manifold with boundary $(f\circ p)^{-1}(]-\infty,a_{k}])$ has bounded geometry with constants appearing in Def.~\ref{manifold of bounded geometry} independent of $X$.

From (\ref{induced geometry}) and Lemma~\ref{uniform collar} above one infers that the bounds of the geometry of $\bar Z_{k}$ are uniform in $k$. Moreover, the coarea formula proves that  the level sets $\partial Z_{k}= f^{-1}(a_{k})$, $k\in \NN$, have volume which converge to zero.
\begin{remark} Following the above lines, one observes that an adaptation of the arguments given by Dafermos \cite{Daf} reproves and simplifies the proof of  the Theorem on Good Choppings of Riemannian manifolds (see \cite[Thm.~2.1]{CheGro85-characteristic}, and \cite{CheGro91-chopping}) with bounded geometry.
\end{remark}

\subsubsection{Proof of  convergence} \label{proof of convergence for special exhaustion}
The bounds on the geometry of the manifolds $p^{-1}(\bar Z_{k})$, $p^{-1}(\partial Z_{k})$ and $X\setminus p^{-1}(Z_{k})$ are uniform in $k$ and $\lim_{k\to +\infty}\Vol(\partial Z_{k})= 0$. The various Bergmann kernels are therefore uniformly bounded in $k$, hence
\begin{align}
\lim_{k\to +\infty} b_{i(2)}(X\setminus p^{-1}(Z_{k}),E)= 0\, , &\quad \text{ and } \quad \lim_{k\to +\infty} b_{i(2)}(\partial p^{-1}(Z_{k}),E)= 0\,.
\end{align}

One notices that the last equality will not be used in the proof of the convergence theorem.

The Mayer-Vietoris sequence and duality induced by the Hodge star operator imply \begin{align} \lim_{k\to +\infty} b_{i(2)}(p^{-1}(Z_{k}), E)= \lim_{k\to +\infty} b_{i(2)}(p^{-1}(Z_{k}),\partial p^{-1}(Z_{k}),E) = b_{i(2)}(X,E)\, , \quad\text{  and  } \nonumber \\
 \quad\lim_{k\to +\infty} b_{i(2)}(X\setminus p^{-1}(Z_{k}), E) =\lim_{k\to +\infty} b_{i(2)}(X\setminus p^{-1}(Z_{k}),p ^{-1}(\partial Z_{k}),E) =0\,.\end{align}

Indeed, let $K(Z_{l},Z_{k})$\footnote{ As in Prop.~\ref{Mayer Vietoris}, one can see that  $K(Z_{l},Z_{k})$ is isomorphic to Schick's complex $A(p^{-1}(\bar Z_{l}-Z_{k}), p^{-1}(\partial Z_{k}),E)$, whose domains are the closures of smooth $E-$valued forms with compact support in $p^{-1}(\bar Z_{l}-Z_{k})$ and relative condition on $p^{-1}(\partial Z_{k})$. }
be the complex defined by the kernel of \begin{align}H^{\tbullet}(p^{-1}(Z_{l}),\nabla_{ \max},E)\to H^{\tbullet}(p^{-1}(Z_{k}),\nabla_{ \max},E)\,.\end{align}
Diagram chasing proves that the following sequence is well defined and exact \begin{align}0\to H(X-p^{-1}(Z_{l}),\nabla_{\min})\to  H(X-p^{-1}(Z_{k}),\nabla_{\min})\to K(Z_{l},Z_{k})\to 0\,.\end{align}

Let $W^{\tbullet}_{l}$ (resp. $W_{\infty}^{\tbullet}$) be the image of $\underline{H}^{\tbullet}_{(2)}(p^{-1}(Z_{l}),E)$ (resp.  $\underline{H}^{\tbullet}_{(2)}(X,E)$) in $\underline{H}^{i}_{(2)}(p^{-1}(Z_{k}),E)$. We claim that $\lim_{k\to +\infty}\dim_{\Gamma}\bar W^{\tbullet}_{l}/\bar W^{\tbullet}_{\infty}=0$: 
The above exact sequence proves that $\bar W^{\tbullet}_{l}/\bar W^{\tbullet}_{\infty}$ is isomorphic to the closure of the image of $\underline{H}^{\tbullet}_{(2)}(X\setminus p^{-1}(Z_{l}),E)$ in $\underline{H}^{\tbullet}_{(2)}(X\setminus p^{-1}(Z_{k}),E)$. One concludes that $\lim_{l}\dim_{\Gamma}\underline{H}^{\tbullet}_{(2)}(X\setminus p^{-1}(Z_{l}),E)=0$, in particular $\lim_{l}b_{.(2)}(p^{-1}(Z_{k}),p^{-1}(Z_{l}),E)=b_{.(2)}(p^{-1}(Z_{k}),E)$.

\subsection{Proof of the homotopy invariance} One can follow the lines of Cheeger and Gromov \cite{CheGro85-characteristic} in the proof of theorem 1.1, once a good exhaustion sequence is known to exist.
\end{proof}

\section{applications}
We consider the particular case of homotopy invariance. Recall Definition~$(\ref{l2 direct image})$ of the locally constant sheaf $\p(E)\to Y$ associated to a Hermitian bundle equipped with a flat connection. Assume the inclusion $\bar Y_{0}\to Y$ of a compact submanifold with boundary is a homotopy equivalence. 
There exists an exhaustion by submanifolds with boundary $(\bar Y_{j})_{j\in \NN}$ such that $\bar Y_{0}\to \bar Y_{k}$ is a homotopy equivalence and $\bar Y_{j}$ is the support of a subcomplex $K_{j}$.
 Hence \begin{align}H^{i}(\bar Y_{k},\p(E))\to H^{i}(\bar Y_{0},\p(E)) &\quad \text{(sheaf cohomology), and}\\
 H^{i}_{(2)}(p^{-1}(K_{k}),E))\to H^{i}_{(2)}(p^{-1}( K_{0}) ,E)  &\quad \text{(simplicial }l^{2}-\text{cohomology)}
  \end{align} are isomorphisms, and the induced morphisms between reduced cohomology groups \begin{align}\underline{H}^{i}_{(2)}(p^{-1}(K_{k}),E))\to \underline{H}^{i}_{(2)}(\bar Y_{0}),E) \end{align} are weak isomorphisms. The
  DeRham Theorem implies that $b_{i(2)}(p^{-1}(Y_{k}),E)=b_{i(2)}(p^{-1}(Y_{0}),E)$. Convergence yields that
\begin{align}\dim_{\Gamma}\H^{i}_{\nabla(2)}(X,E)=\dim_{\Gamma}\underline{H}^{i}_{(2)}(p^{-1}(Y_{0}),E)\nonumber\\=\dim_{\Gamma}\underline{H}^{i}_{(2)}(p^{-1}(K_{0}),E)=\dim_{\Gamma}H^{i}(\bar Y_{0},\p(E)) \,.\end{align}
The first two spaces are analytic and depend on the metric but are subject to Hodge duality. The third space depends on the simplicial structure, is combinatorial and subject to duality, the last space is computed from sheaf theory. One uses the generalized dimension function of Lück \cite{Luc}.

The cohomology of the locally constant sheaf $\p(E)\to Y$ is invariant under homotopy, hence $ H^{i}(Y,\p(E))\to H^{i}(\bar Y_{0},\p(E))$ is an isomorphism. We refer to \cite[Sec.~2.4]{Din2013} or \cite{Luc} for the notion of isomorphism modulo a torsion subcategory, which is used in the following.

\begin{theorem}
Let $X\to Y$ be a $\Gamma-$covering of a complete manifold of bounded volume such that $X$ is of bounded curvature and positive injectivity radius. Let $(E,h,\nabla_{h},\nabla)\to Y$ be a hermitian bundle equipped with a flat connection. We denote by $(E,h,\nabla)\to X$ its pull back to $X$ and assume it satisfies Property~$AC^{0}$ (cf.\ (\ref{property A})).
Assume that $Y$ is of finite topological type. Then the natural morphism
\begin{align}
\H^{i}_{\nabla(2)}(X,E)\to H^{i}(Y,\p(E))
\end{align}
is an isomorphism mod $\tau_{\dim}$ or $\tau_{\U(\Gamma)}$.
\end{theorem}
Note that no growth assumption is assumed in the space on the right, hence the isomorphism class of this space is independent of metrics on $Y$ and $E\to Y$.
\begin{proof}
The above morphism is in fact injective: Let $\alpha\in \H^{i}_{\nabla(2)}(X,E)$ (associated harmonic space), whose image in $H^{i}(Y,\p(E))\otimes \U(\Gamma)$ vanishes. The section $\alpha$ vanishes in $H^{i}(Y_{n},\p(E))\otimes \U(\Gamma)$, where $(Y_{n})_{n}$ is an exhausting sequence of $Y$ by relatively compact open submanifolds homotopically equivalent to $Y$. The Mayer-Vietoris sequence (Proposition~\ref{Mayer Vietoris}) implies that $\alpha\otimes 1_{\U(\Gamma)}$ belongs to $\underline{H}^{i}_{(2)}(X\setminus p^{-1}(\bar Y_{n}), \partial p^{-1}(Y_{n}), E)\otimes \U(\Gamma)$. Hence \begin{align}\dim_{\Gamma} [\VN(\Gamma) \alpha]\leq \lim_{n}\dim_{\Gamma} \underline{H}^{i}_{(2)}(X\setminus p^{-1}(\bar Y_{n}), \partial p^{-1}( Y_{n}), E)=0\,.\end{align} This proves the assertion. However, $b_{i(2)}(X,E)=\dim_{\Gamma} H^{i}(Y,\p(E))$, hence the morphism is almost surjective that is $\H^{i}_{\nabla(2)}(X,E)\otimes \U(\Gamma)\to H^{i}(Y,\p(E))\otimes\U(\Gamma)$ is onto.
 \end{proof}

Hodge duality implies $b^{i}_{(2)}(X,E)=b^{n-i}_{(2)}(X,E)$. Recall that a Stein manifold of complex dimension $n$ has the homotopy type of a $CW-$complex of real dimension $n$.  One deduces the following vanishing results:
\begin{corollary}
With the above hypotheses, assume that $Y$ has the homotopy type of a finite $CW-$complex of dimension $k$, then $\dim_{\Gamma}H^{i}(Y,\p(\underline{E}))$ (cf.\ \ref{l2 direct image}) is non vanishing only in the range  $\dim_{\RR}Y-k\leq i\leq k$. In particular:
\begin{itemize}
\item[i)] if $2k<\dim Y$, then any Galois covering of positive injectivity radius is $l^{2}-$acyclic.
\item[ii)] Assume $Y$ is a Stein manifold of dimension $n$, and let $(E,\dbar,\theta,h)\to Y$ be a harmonic Higgs bundle with bounded Higgs field then $\H^{l}_{\nabla(2)}(X,E)$ is vanishing if $l\not=n$.
\item[iii)] Let $(\bar Y,\omega_{0})$ be a compact Kähler manifold, $D$ a normal crossing divisor such that $Y=\bar Y\setminus D$ is Stein. Let $p:X\to \bar Y\setminus D$ be a Poincaré covering.
Let $(E,\nabla)\to \bar Y\setminus D$ be a semi-simple unipotent flat bundle with associated local system $\underline{E}=\Ker(\nabla)$. Then $H^{i}(Y,\p(\underline{E}))=0$ if $i\not=\dim_{\CC}Y$.
\end{itemize}
\end{corollary}

We recall that Section~\ref{ddbar lemma for higgs bundles} established an isomorphism mod $\tau_{\dim}$ or $\tau_{\U(\Gamma)}$ between  DeRham cohomology and Dolbeault cohomology of $(E,\theta)$.

This theorem applies in particular to a unitary representation that is a holomorphic hermitian bundle with vanishing hermitian curvature. For such a bundle over a Kähler manifold, the Hodge decomposition is compatible with the harmonic projection, hence
\begin{align}
\oplus_{p+q=i}\H^{(p,q)}_{\dbar (2)}(X,E)\to H^{i}(Y,\p(E))
\end{align}
is an isomorphism mod $\tau_{\dim}$ or $\tau_{\U(\Gamma)}$. In particular, if $Y$ has the homotopy type of a $CW-$complex of dimension $k$, then the Dolbeault $l^{2}-$cohomology groups $\H^{(p,q)}_{\dbar (2)}(X,E)$ of a flat hermitian holomorphic bundle are non-vanishing only in the range $[2n-k,k]$ and the Galois $\partial\dbar-$lemma is true.
 \subsection{Application to compact Kähler manifolds}

Let $p:X\to Y$ be a covering of a compact manifold. Then the $l^{2}-$cohomology may be developed using sheaf theory. Campana-Demailly \cite{CamDem} and Eyssidieux \cite{Eys} defined an exact functor $\F\to \p\F$ (see also Definition~\ref{l2 direct image} of the present article) from coherent sheaves over $Y$ to sheaves over $Y$ such that, for a sheaf $\F$ associated to a hermitian holomorphic vector bundle $(F,h)\to Y$, the cohomology group $H^{i}(Y,\p\F)$ of the sheaf $\p\F\to Y$ is naturally isomorphic to the Dolbeault $l^{2}-$cohomology groups $H^{i}_{\dbar(2)}(X,p^{*}(F))$. Exactness of the functor implies that the $l^{2}-$hypercohomology of a Higgs bundle $(E,\theta)$ is the limit of two spectral sequences, well known in the absolute case $X=Y$. The first, $$H^{p}_{\dbar (2)}(X,\Omega^{q}\otimes E)\Rightarrow \HH^{p+q}_{(2)}(X,(E,\theta),\omega)\,,$$ defined in Section~\ref{ddbar lemma for higgs bundles}, does not depend upon compactness.

  The second $$H^{i}(Y,\p\H^{j}(\theta))\Rightarrow \HH^{i+j}(Y,\p(E,\theta))$$ relates the degeneracy locus of $\theta$, through the homology sheaf $\H^{j}(\theta):=\H^{j}(\ldots\overset{\theta}{\to}E\otimes \Omega^{j}\overset{\theta}{\to}\ldots)$, to the global $l^{2}-$cohomology groups. The dimensions of the support of the homology sheaves $\H^{j}(\theta)$ restrict therefore the non-vanishing range of the $l^{2}-$cohomology of the harmonic Higgs bundle. We can state a theorem for flat unitary bundles and holomorphic one forms since, good estimates on the dimension of the homology locus  exist (see \cite{GreLaz}, \cite{Eis}): Let $\theta$ be a holomorphic one form, and let $Z(\theta)$ be the zero locus of $\theta$ then, for all $x\in Y$, then $\H^{j}(\theta)_{x}$ vanishes, if $j<\codim Z(\alpha)_{x}$ and it is not equal to zero, if $j=\codim Z(\alpha)_{x}$.

\begin{corollary} Let $p:X\to Y$ be a Galois covering of a Kähler manifolds, let $\theta$ be a holomorphic one form and
$(E,\nabla,h)\to Y$ be a flat unitary bundle.
Then the complex $$\ldots\overset{\theta}{\to}H^{i}_{\dbar(2)}(X,E\otimes\Omega^{j-1})\overset{\theta}{\to}H^{i}_{\dbar(2)}(X,E\otimes \Omega^{j})\overset{\theta}{\to}
H^{i}_{\dbar(2)}(X,E\otimes \Omega^{j+1})\overset{\theta}{\to}\ldots$$
is exact  mod $\tau_{\dim}$ or $\tau_{\U(\Gamma)}$, in degree $(i,j)$ such $|i+j -n|>\dim Z(\alpha)$ 
\item Assume moreover that $\codim Z(\alpha)=n$, then $H^{n}_{(2)}(X,\Ker(\nabla))\not=0$.
\end{corollary}

Now we assume the form $p^{*}(\theta)$ is exact. Let $\pi$ be the projection onto the harmonic space. Then as noticed in Jost-Zuo \cite{JosZuo00}, the mapping $\H^{i}(X,E\otimes \Omega^{j})\overset{\pi\theta}{\longrightarrow} \H^{i}(X,E\otimes \Omega^{j+1})$ vanishes (here $E$ may be any hermitian bundle). Indeed let $\theta =du$.
Then $u$ has linear growth. Let $\chi$ be a smooth positive function such that $\chi_{|]-\infty, 1]}=1$, $\chi_{|[2,+\infty[}=0$, and let $d(0,.)$ be the distance from a fixed point $0\in X$. One checks that for any harmonic form $\alpha\in \H^{i}_{\dbar(2)}(X,E\otimes \Omega^{j})$, \begin{align}\lim_{\epsilon\to 0 }d[\chi(\epsilon d(0,.)) u .\alpha]=\theta.\alpha\end{align} in $l^{2}$.
Therefore:
\begin{theorem}\label{untwisted Green-Lazarsfeld theorem} In the above situation assume that $p^{*}(\alpha)$ is exact, then $\H^{i}_{\dbar(2)}(X,E\otimes \Omega^{j})$ is vanishing, if $|i+j-n|> \dim Z(\alpha)$. In particular $\H^{k}_{(2)}(X,\Ker(\nabla))$ is vanishing, if $|k-n|>\dim Z(\alpha)$.
\begin{itemize}
\item[(i)] Assume that $\alpha$ is nowhere vanishing, then $p:X\to Y$ is $l^{2}-$acyclic for any unitary bundle.
\item[(ii)] Assume that $\dim Z(\alpha)=0$ then $H^{i}_{(2)}(X,\Ker(\nabla))$ is vanishing, if $i\not=n$ and non vanishing for $i=n$.
\end{itemize}
\end{theorem}

Using the decomposition in Fourier series of harmonic forms on Abelian covering, one recovers the following generic vanishing theorem of Green and Lazarsfeld \cite{GreLaz}:
\begin{corollary}
Let $\alpha$ be a holomorphic form on a compact Kähler manifold $Y$ and let $E\to Y$ be a unitary holomorphic vector bundle. Then $H^{i}(Y,E\otimes\Omega^{i}\otimes L)=0$ for a generic flat unitary line bundle, if $|i+j-n|>\dim Z(\alpha)$.
\end{corollary}
Compact Kähler manifolds supporting a nowhere vanishing holomorphic form are studied in \cite{PopSch}. In the projective case, the restriction of this form to generic hyperplane section of such manifolds has discrete zero set.
\subsection{ }
The analogous theorem for the Koszul complex of a closed holomorphic one form over a non-compact manifold does not hold: assume that $\theta$ is the restriction of a holomorphic one form, still denoted $\theta$, on a curve $\bar Y$, and let $Y=\bar Y\setminus Z( \theta)\cup S'$, $p:X\to Y$ be the universal covering map. Then the homology sheaves $\H^{i}(\theta)$ are trivial, nevertheless the $l^{2}-$cohomology groups $\H^{0,1}_{(2)}(X,E)\simeq \H^{1,0}_{(2)}(X,E^{*})$ of a flat unitary bundle are non vanishing as soon as $s$, the cardinal of $Z(\theta)\cup S'$, is strictly bigger than $2-2g$. Indeed Theorem~\ref{characteristic integral for a tame harmonic bundle} implies that $\chi_{(2)}(X,E)=r(1-g-\frac{s}{2})$. But the maximum principle implies that $H^{0}_{(2)}(X,E)$ vanishes when $E$ is a flat unitary bundle on a manifold of bounded geometry $X$. 

As in the compact case, the first spectral sequence degenerates at $E_{2}$, while for the second, 
the vanishing of the homology sheaves $\H^{i}(\theta)$ does not imply vanishing of $\Ker(\theta)/\Ran(\theta)$, where now $\theta$ acts on Sobolev spaces.

The study of $(E,\theta)\to Y$, where $Y=\bar Y\setminus D$, and $\theta$ is a logarithmic one form with pole in $D$ is more involved: The form ${dz}/{z}$ is not bounded in the Poincaré metric and results of \cite{Ara97,Dim10} show that the logarithmic Hodge to DeRham spectral sequence $H^{q}(\bar Y,\Omega^{p}_{\bar Y}{log\,D})\abupt \HH^{p+q}(\bar Y, (\Omega^{\tbullet}_{\bar Y}({\log\,D}),\theta))$ does not  degenerate at $E_{2}$ in general.

\end{document}